\documentclass[reqno]{amsart}

\usepackage{amsthm}
\usepackage{amssymb}
\usepackage{mathrsfs}

\usepackage{xcolor}
\newenvironment{Assumptions}% Definition of assumptions
{%
\setcounter{enumi}{0}

\begin{enumerate}}%
{\end{enumerate} }

\newenvironment{Assumptions1}% Definition of assumptions
{%
\setcounter{enumi}{0}

\begin{enumerate}}%
{\end{enumerate} }

\newenvironment{Assumptions3}
{
	\setcounter{enumi}{0}

	\begin{enumerate}}
	{\end{enumerate} }
 \newcommand{\grad}{\nabla}

\newcommand{\rn}{\mathbb{R}^{N}}
\newcommand{\zn}{\mathbb{Z}^{N}}
	\newcommand{\Z}{\mathbb{Z}}

\newcommand{\A}{\alpha}
\newcommand{\bA}{\bar\alpha}

\newcommand{\I}{\mathcal{I}}
\newcommand{\J}{\mathcal{J}}

\newcommand{\bj}{\mathbf{j}}

\newcommand{\Eta}{\eta^{\alpha}}

\def\N{\mathbb{N}}

\usepackage[colorlinks]{hyperref}

\newtheorem{thm}{Theorem}[section]
\newtheorem{lem}[thm]{Lemma}
\newtheorem{prop}[thm]{Proposition}
\newtheorem{cor}[thm]{Corollary}

\theoremstyle{definition}

\newtheorem{rem}[thm]{Remark}

\numberwithin{equation}{section}

\newcommand{\be}{\color{blue}}

\newcommand{\ee}{\normalcolor}

\begin{document}

\title[Precise bounds for fractional HJB equations]
{Precise error bounds for numerical approximations of fractional HJB equations}

\author{Indranil Chowdhury}
\address{Indian Institute of Technology - Kanpur, India}
\curraddr{}
\email{indranil@iitk.ac.in}
\thanks{I.C. was supported by ERCIM ``Alain Bensoussan" Fellowship programme and INSPIRE faculty fellowship (IFA22-MA187).}

%    author two information
\author{Espen R. Jakobsen}
\address{Norwegian University of Science and Technology, Norway}
\curraddr{}
\email{espen.jakobsen@ntnu.no}
\thanks{E.R.J. received funding from the Research Council of Norway under Grant Agreement No. 325114 “IMod. Partial differential equations, statistics and data: An interdisciplinary approach to data-based modelling”.}


%Indranil: {45K05, 46S50, 49L20, 49L25, 91A23, 93E20}

% 35J60  	Nonlinear elliptic equations
% 35J70  	Degenerate elliptic equations
% 49L25  	Viscosity solutions to Hamilton-Jacobi equations in optimal control and differential games
% 49L12  	Hamilton-Jacobi equations in optimal control and differential games
% 49L20  	Dynamic programming in optimal control and differential games
% 49M25  	Discrete approximations in optimal control
% 65M06  	Finite difference methods for initial value and initial-boundary value problems involving PDEs
% 65M15  	Error bounds for initial value and initial-boundary value problems involving PDEs
% 65R20  	Numerical methods for integral equations
% 65N06  	Finite difference methods for boundary value problems involving PDEs
% 65N15  	Error bounds for boundary value problems involving PDEs
% 65N12  	Stability and convergence of numerical methods for boundary value problems involving PDEs
% 34K37  	Functional-differential equations with fractional derivatives
% 35R11  	Fractional partial differential equations
% 35J70  	Degenerate elliptic equations
% 93E20  	Optimal stochastic control
% 65D32  	Numerical quadrature and cubature formulas
% 45K05  	Integro-partial differential equations


\date{}

\dedicatory{}

\begin{abstract}
We prove precise rates of convergence for  monotone approximation schemes of fractional and nonlocal Hamilton-Jacobi-Bellman (HJB) equations. We consider  diffusion corrected difference-quadrature schemes from the literature and new approximations based on powers of discrete Laplacians, approximations which are (formally) fractional order and 2nd order methods. It is well-known in numerical analysis that convergence rates depend on the regularity of solutions, and here we consider cases with varying solution regularity: (i) Strongly degenerate problems with Lipschitz solutions, and (ii) weakly non-degenerate problems where we show that solutions have bounded fractional derivatives of order $\sigma\in(1,2)$. Our main results are optimal error estimates with convergence rates that capture precisely both the fractional order of the schemes and the fractional regularity of the solutions. 
For strongly degenerate equations, %of fractional order, 
these rates improve earlier results. For weakly non-degenerate %nonlocal 
problems of order greater than one, the results are new. Here we show improved rates 
compared to the strongly degenerate case, rates that are always better than $\mathcal{O}\big(h^{\frac{1}{2}}\big)$.
\end{abstract}

\maketitle

\tableofcontents

\section{Introduction}
In this paper we prove 
precise rates of convergence for  monotone approximation schemes of fractional and nonlocal Hamilton-Jacobi-Bellman (HJB) equations. Weakly non-degenerate problems are studied, and we give error bounds with convergence rates capturing both the fractional orders of accuracy of schemes and regularity of solutions.

HJB equations are fully nonlinear possibly degenerate PDEs from optimal control theory with a large number of applications in engineering, science,  economics etc. \cite{Be:book,FS:book,OS:book,Ha:book}. In this paper we focus on the following nonlocal version:
\begin{align}\label{eqn:main}
 &  \sup_{\A\in\mathcal{A}} \left\{ f^{\A}(x) +c^{\A}(x) u(x)  - \I^{\A}[u](x) \right\}  = 0  \mbox{\quad in} \quad \rn, 
 \end{align}
where $N\in\N$, $\mathcal{A}$ is a compact metric space, the integral operator 
\begin{align} \label{n_local:term}
\I^{\A}[\phi](x) := \int_{\rn\backslash\{0\}} \Big( \phi(x+\eta^{\A}( z)) - \phi(x) -\eta^{\A}(z)\cdot \nabla_x\phi(x)\,1_{|z|<1}\ee\Big) \, \nu_{\A}(dz),
\end{align}
and the L\'evy measure $\nu_{\A}$ is  nonnegative with $\int |z|^2\wedge 1\ \nu_\alpha(dz)<\infty$. 
The operator $\I^{\A}$ is a fractional (convection-)diffusion operator of maximal order $\sigma \in (0,2)$.
%, see assumption \ref{A6}.  
Our assumptions encompass fractional Laplacians, 
tempered operators from finance, and any other generator of a pure jump L\'evy process. The coefficients are bounded, continuous, and $x$-Lipschitz uniformly in $\alpha$, see Section \ref{sec:assump_wellposed} for all assumptions. 

Equation \eqref{eqn:main} is the dynamic programming equation for the value function of an infinite horizon optimal stochastic control problem \cite{OS:book,Ha:book}:
$$u(x)=\inf_{\alpha_\cdot\in \mathcal A_{\text{ad}}} \int_0^\infty e^{-\int_0^s c^{\alpha_r}(X_r)dr}f^{\alpha_s}(X_s)\,ds,$$
where $\mathcal A_{\text{ad}}$ is the set of admissible controls and the (controlled) process $X_s$ is given by a L\'evy driven SDE \cite{Ap:Book,CT:Book} of the form
\begin{align}\label{SDE}dX_s = x + b^{\alpha_s}\,ds+\int_{|z|<1} \eta^{\alpha_s}(z)\tilde N(\alpha_s,dz,ds)+\int_{|z|\geq1} \eta^{\alpha_s}(z)N(\alpha_s,dz,ds).
\end{align}
The (compensated) Poisson random measure $N$ ($\tilde N$) of the driving L\'evy process %(and its compensated version $\tilde N$) 
has an intensity/L\'evy measure $\nu_a$ such that $\mathbb E\big[ N(a,B,(0,t))\big]=\nu_a(B)t$\footnote{This is the expected number of jumps $z\in B$ of the L\'evy process up to time $t$ \cite{Ap:Book,OS:book}.} for Borel sets $B\not\ni 0$ and $t>0$. For simplicity we focus on pure (jump) diffusion processes with $b=0$ and HJB equations of the type \eqref{eqn:main}, but at the end of the paper we give results for more general equations with $b\neq0$. %drift-diffusion case. 

The operator $\I^{\A}$ will always be at least degenerate elliptic under our assumptions. When we also assume (loosely speaking) that
$$\big(\tfrac{d\nu_\A}{dz}, \eta^\A(z)\big) \to \big(\tfrac{c_1}{|z|^{N+\sigma}}, c_2z\big)\qquad\text{as}\qquad z\to0,$$
$\I^{\A}$ will be non-degenerate and uniformly elliptic.
We refer to \ref{C1} and \ref{A6} for precise assumptions. The HJB equation \eqref{eqn:main} is {\it strongly degenerate} if the operators $\I^{\A}$ are degenerate for every $\alpha$,\footnote{E.g. there could be no diffusion in some directions, or the operator could be a $0$ order operator with bounded L\'evy measure. There could be different degeneracies for different $\alpha$'s.} and it is {\it weakly non-degenerate} if there is at least one $\alpha$ for which $\I^{\A}$ is elliptic/non-degenerate. Obstacle problems for elliptic operators are examples of weakly non-degenerate problems \eqref{eqn:main}, and they are known to have non-smooth solutions (at the contact set).
The correct (weak) solution concept for this type of problems is viscosity solutions \cite{JK05,JK06,BI08}. Wellposedness, regularity, asymptotics, approximations, and other properties of viscosity solutions for nonlocal PDEs has been intensely studied in recent years. Regularity in the strongly degenerate case comes from comparison type of arguments and typically gives preservation of the regularity of the data \cite{JK06}. Solutions can then be at most Lipschitz continuous. In non-degenerate cases there is a regularizing effect. The regularity theory has mostly been developed for uniformly elliptic/parabolic problems, and the huge literature includes seminal works of Caffarelli and Silvestre \cite{CafSilv09, CafSilv11}. In the weakly non-degenerate case there are few results, %a part from ours, we 
and most relevant for us (our inspiration) is \cite{D-Krylov05}  for local problems.
We show here that weakly non-degenerate problems of order $\sigma\in(1,2)$ have solutions with bounded fractional derivatives of order $\sigma$.\footnote{We assume that the data is semiconcave to achieve this.} Hence solutions are more smooth than in the strongly degenerate case. Independently, similar type of regularity results have been obtained in the very recent preprint \cite{RW:2023} on nonlocal obstacle problems.

There is a huge literature on numerical methods for local HJB equations including finite differences, semi-Lagrangian, finite elements, spectral, Monte Carlo, 
%BSDE Monte-Carlo, Deep-BSDE 
and many more, see e.g. \cite{Crandall-lions,DK:book,FF:book,BS91,Le00,BH:2001,BZ03,Deb_Jak:13,SS:2013,BT2004,EHJ2017}.
For fractional and nonlocal problems, there is the added difficulty of
discretizing the fractional and nonlocal operators in a monotone, stable, and
consistent way. These operators are singular integral operators, and
can be discretized by quadrature after truncating the singular part and correct with a
suitable second derivative term. This diffusion corrected
approximation was introduced on the level of processes in \cite{AR01}
and then for linear PDEs e.g. in \cite{Cont-2005} in connection with
difference-quadrature schemes, see also \cite{JKL08,BJK1,Oberman2016}. In the setting of
HJB equations, it was introduced in
\cite{JKL08,C-Jakobsen2009,BJK1}
with further developments in e.g.
%and then various extensions have been given in e.g. 
\cite{BCJ1, Coc-Risebro16, Reisinger2021, Roxana-Reisinger2021}. We will give new results for this approximation here, and focus on a version based on semi-Lagrangian type approximations \cite{Camilli-Falcone,Deb_Jak:13} of the nonlocal operators \cite{C-Jakobsen2009}. Another way of discretizing certain
fractional operators, is via subordination: When the operator is a
fractional Laplacian, it can be discretized by a (fractional) power of
the discrete (FDM) Laplacians which can be seen as a quadrature rule with
explicit weights \cite{Ciaurri-Stinga-02}. While 
the diffusion corrected approximation has fractional order accuracy,
the power of discrete Laplacian approximation is always of second
order and faster when the order of the equation is close to $2$. This last
approximation has previously been used to solve linear and
porous medium equations \cite{EJT18b,BP15}. In this paper, we will explain how
it can be used to solve HJB equations and provide error bounds.

The main focus of the paper is on precise error bounds for the schemes and regularity settings mentioned above, especially the weakly non-degenerate case. In numerical analysis it is well-known that such bounds must depend on both the accuracy of the method and the regularity of solutions. In our fractional setting, both of these may be fractional, and previous results are either not optimal or lacking. While linear, local, and smooth problems can analyzed in a rather simple and classical way \cite{RM:book}, error analysis is more complicated in our fully nonlinear and non-smooth setting. There are
two main approaches:\footnote{In the
  uniformly elliptic case, there are other methods \cite{CafSoug, Krylov:2015, %Tura15a,
   Tura15b}. These results are not explicit nor optimal, but they apply
  also to nonconvex problems. See also \cite{Jakobsen2004, Bon-Zid06, Jakobsen2006}.} (i) The `doubling of
variables' technique for fully nonlinear equations of
1st order \cite{Dolcetta-Falcone, Crandall-lions, Souga} or fractional
order less than 2 \cite{BCJ1,Coc-Risebro16}; 
 %applicable to 1st order drift/convection operators \cite{Dolcetta-Falcone, Crandall-lions, Souga} and general nonlocal diffusion operators of order less than two \cite{BCJ1}; 
 and (ii) the 'shaking of coefficients' method for convex HJB
 equations of 2nd order \cite{BJ02,BJ05,  D-Krylov05,
   Krylov:1997,Krylov:2000yy,Krylov:2005lj} or fractional order %BJK3,
 \cite{BJK2,BJK1,JKL08}.  
 
 The `shaking of coefficients' method,
 originally introduced by Krylov, relies on constructing smooth
 subsolutions of both the equation and the scheme which can then be used to
 get one-sided error estimates via the comparison principle and local
 consistency bounds. If precise regularity results for both the scheme and the equation are known, along with sharp consistency bounds, the method produces optimal rates. We refer to \cite{D-Krylov05,Jakobsen2003, Krylov:2005lj} for local
 2nd order problems and \cite{BJK2,C-Jakobsen2009, JKL08} for nonlocal
 problems. If regularity of the scheme is not known (this is difficult in general), sub-optimal rates can still be proved \cite{BJ02,BJ05,BJK1}, and these latter bounds holds for a very large class of monotone schemes.
 Note that the 'shaking of coefficients' method has the advantage that it can handle arbitrary high order error equations and therefore also higher order methods, while the 'doubling of variables' method only work optimally for schemes with (at most) 2nd order truncation errors.  
 For nonlocal HJB equations, most of the progress on 
 optimal error bounds for monotone schemes have addressed
 bounded (non-singular) integral operators %BJK3
 \cite{BJK2, JKL08}. Non-optimal bounds for problems with singular
 operators can then be obtained after first approximation by bounded
 operators. Without this approximation step, sub-optimal rates have
 been obtained in \cite{BJK1} for singular integral operators.
\medskip

\noindent \textbf{Our main contributions:} 

\medskip

\noindent (a) A rigorous {\em error analysis for monotone approximations of weakly non-degenerate problems} is developed in Section \ref{sec_prf_gen}. This is new and based on the ``method of shaking the coefficients". The proof amounts to extending the analysis of \cite{D-Krylov05} to nonlocal/fractional equations and schemes. Our setting is more involved and technical. The main challenges are related to the {\em fractional} approximation, regularization, and regularity results needed -- both for the equation and the scheme. As opposed to previous nonlocal results, we cannot use standard mollifiers for regularization but crucially need fractional heat kernels. 
For the schemes, the results are discrete and contain error terms, and a very careful analysis is needed to get optimal results. 
\medskip

\noindent (b) {\em $C^{1,\sigma-1}$-regularity results for weakly non-degenerate HJB-equations} of order $\sigma\in(1,2)$ given in Theorem \ref{thm:regularity:vis_soln:wdeg:nonlocal}. These are natural extensions to nonlocal/fractional problems of the $W^{2,\infty}$ results of \cite{D-Krylov05}. They seem to be new for equations of fractional order (but see also \cite{RW:2023}) and are of independent interest.
Our proof is based on uniform estimation of approximate fractional derivatives based on semi-concavity estimates and exploitation of weak non-degeneracy followed by an application of regularity results for linear problems in \cite{Oton2016}. We also need and prove discrete versions of such results. 
\medskip

\noindent (c) Precise error bounds for {\em diffusion corrected difference-quadrature schemes} in Section \ref{monotone scheme}. Under various assumptions, we roughly speaking show that if $\sigma$ is the order of equation \eqref{eqn:main}, $u$ its solution, and $u_h$ the solution of scheme, then
\begin{align*}
\|u- u_{h}\|_{L^\infty} \leq \left\{
       \begin{array}{ll}
            C \, h^{\frac{1}2(4-\sigma)} \quad &\mbox{when solutions are smooth ($C^4_b$)},      \\[0.2cm]
      C \, h^{\frac{\sigma}{4+ \sigma}(4-\sigma)} \quad &\mbox{in the weakly non-degenerate case and $\sigma>1$},\\[0.2cm]      C \, h^{\frac{1}{4+\sigma}(4-\sigma)} \quad &\mbox{in the strongly degenerate case or when $\sigma\leq1$}. 
       \end{array}
  \right.
\end{align*}
Here the accuracy is a decreasing function of $\sigma$, which is reflected in decreasing rates in $\sigma$ when the regularity is fixed (strongly degenerate and smooth cases). In the weakly non-degenerate case, regularity is increasing with $\sigma$ and so are the rates despite decreasing accuracy. Rates are higher when solutions are more regular and maximal in the smooth case. 
%The rates are decreasing functions of $\sigma$ (following the accuracy), they are higher when solutions are more regular and maximal in the smooth case. 
These results are sharper than previous results \cite{BJK2, BJK1,BCJ1} in the strongly degenerate case, and new in the weakly non-degenerate case where %. In the latter case 
the rate increases from $\frac35$ at $\sigma=1$ to $\frac23$ in the limit as $\sigma\to2$.\!\!\!
\medskip

\noindent (d) New {\em approximations based on powers of discrete Laplacians} are introduced in Section \ref{sec:fraclap} for HJB equations with fractional Laplacians, $\mathcal I^\alpha[\phi]=-a^\alpha (-\Delta)^{\frac\sigma 2}\phi$.
 These problems are always weakly non-degenerate, and we prove precise error bounds,
\begin{align}
\|u- u_{h}\|_{L^\infty} \leq \left\{
       \begin{array}{ll}
      C h^{\frac{1}{2}} \quad &\mbox{for} \quad 0< \sigma \leq 1, \\[0.2cm]
      C h^{\frac{\sigma}{2}} \quad &\mbox{for} \quad 1< \sigma <2 .
       \end{array}
  \right.
\end{align} 
Under our assumptions these rates are  optimal, and as $\sigma\rightarrow 2$, the error bounds approach the $\mathcal{O}(h)$ bound in the local 2nd order case \cite{D-Krylov05}.\footnote{When $\sigma\to2$, problem \eqref{eqn:main} converges by \cite{CJ17} to local the 2nd order problem of \cite{D-Krylov05}.}
\medskip
  
\noindent \textbf{Outline.} The remaining part of this paper is organized as follows: In Section \ref{sec:assump_wellposed} we introduce the notation and assumptions for the strongly degenerate and weakly non-degenerate problems, and give  wellposedness and regularity results for equation \eqref{eqn:main} in both cases.  In Section \ref{monotone scheme} we consider the diffusion corrected difference-quadrature approximations of \eqref{eqn:main} for general nonlocal operators and state our main error bounds. In Section \ref{sec:fraclap} we give the results for approximation based on powers of discrete Laplacians. The proofs of these results are given in Sections \ref{sec_prf_gen} and \ref{sec:prf:flap}. 
In Section \ref{sec:exten} we discuss 
extensions to problems with non-zero drift and more non-symmetric diffusions.    

\section{Strongly and weakly non-degenerate fractional HJB equations}\label{sec:assump_wellposed} \ee
%Prelimineries 

In this section we present the  assumptions on nonlocal HJB equations and give wellposedness and regularity results. We start by introducing some notation. By $C,K$ etc. we mean various constants which may change from line to line, $|\cdot|$ is the euclidean norm, and the norms $\|u\|_0= \sup _{x} |u(x)|$ and $\|u\|_1 = |u|_0 + \sup_{x\neq y} \frac{|u(x)-u(y)|}{|x-y|}$. $C_b(Q)$ is the space of bounded continuous %real valued
functions on $Q\subset \rn$, while $C^{n}(Q)$ and $C^{n, \gamma}(Q)$ for $n \in \N$ and $\gamma \in (0,1]$,  denote the spaces of $n$-th time continuously differentiable functions on $Q$ with finite norms
%$\|\cdot\|_{n}$ and $\|\cdot\|_{n, \gamma}$,  where
$$ \|u\|_{n}= \sum_{j=0}^n \|D^{j}u\|_0 \qquad \mbox{and}\qquad \|u\|_{n,\gamma}= \|u\|_{n} + \sup_{x\neq y} \frac{|D^n u(x)-D^n u(y)|}{|x-y|^{\gamma}}, $$
where $D^n u$ is the ($n$-form of) $n$-th order derivatives of $u$.

\subsection{Assumptions and wellposedness of \texorpdfstring{\eqref{eqn:main}}{}}

First we list assumptions needed for wellposedness and Lipschitz regularity of viscosity solutions of \eqref{eqn:main}. \smallskip
\begin{Assumptions}
\item \label{A1} $\mathcal{A}$ is a separable metric
 spaces, $c^{\A}(x) \geq \lambda> 0$, and $c^{\A}(x),f^{\A}(x)$, and $\eta^{\alpha}(z)$ are continuous in $\A$, $x$,  and $z$. 
 
 \medskip

\item \label{A2} There is a $K>0$ such that % for every $\A$, 
$$  \|f^{\A}\|_{1}+ \|c^{\A}\|_{1} + \|\eta^{\A}\|_0 \leq K\quad\text{for}\quad \A\in \mathcal{A}. $$ 

\medskip

\item \label{A3} There is a $K>0$ such that 
\begin{align*}
 \quad |\Eta(z)| \leq K |z| \quad \mbox{for} \quad |z|<1,\quad \A\in \mathcal{A}.
\end{align*}

\medskip 

 \item\label{A4}  $\nu_{\A}$ is a nonnegative Radon measures on $\rn$ and there is $K>0$ such that 
  \begin{align*}
   \int_{|z|\leq1} |z|^2\nu_{\A}(dz)+ \int_{|z|>1} \nu_{\A}(\,dz) \leq K.
  \end{align*}
\end{Assumptions}
 
In some results we also need symmetry assumptions on the nonlocal terms and upper bounds on the density of the L\'evy measure.
\smallskip
\begin{Assumptions}
\setcounter{enumi}{4}
\item \label{A5} %The L\'evy measure
$\nu_{\A}(dz)\,1_{|z|<1}$ is symmetric for $\A \in \mathcal{A}$. \ee

\medskip
\item\label{A6}
$\nu_{\A}$ is absolutely continuous on $|z|<1$, and 
there are 
$\sigma \in (0,2)$, $M\in \mathbb{N}$, and $C >0$ such that 
  \begin{align*}
  0\le \frac{d\nu_{\A}}{dz} \le \frac{C}{|z|^{M+\sigma}}\qquad\text{for}\qquad|z|<1,\quad \A\in\mathcal{A}.
\end{align*}

  \medskip 

  \item \label{A7}  $\eta^{\A}(-z)= - \eta^{\A}(z) $ \ for \ $|z|<1$ and $\A \in \mathcal{A}$. 
\end{Assumptions}

%\vspace*{0.5cm}
%\noindent  Sometimes we also assume the following condition:
% \begin{Assumptions}
% \setcounter{enumi}{6}
% \item \label{A7}  $\eta^{\A}(-z)= - \eta^{\A}(z) $ for all $|z|<1$ and $\A \in \mathcal{A}$. 
% \end{Assumptions}

\begin{rem}
(a) Under \ref{A3} and \ref{A4}, {\em any} pure jump L\'evy process is allowed as a driver for the SDE \eqref{SDE}. This includes stable, processes, tempered processes, spectrally one-sided process, compound Poisson processes, and most jump processes considered in finance \cite{Ap:Book,CT:Book}.
%Assumptions \ref{A5}-\ref{A7} are not required for wellposedness.
The generators of these processes are $\I^{\A}$. 
\medskip

\noindent(b) Assumption \ref{A6} is a restriction implying that $\I^{\A}$ (which may be degenerate) contains fractional derivatives of orders at most $\sigma$. It can be replaced by a more general integral condition to also cover non-absolutely continuous L\'evy measures,
 $$ \displaystyle r^{-2+\sigma}\int_{|z|<r} |z|^2 d\nu_{\A} + r^{-1+\sigma}\int_{r<|z|<1} |z| d\nu_{\A} + r^{\sigma}\int_{r<|z|<1} d\nu_{\A} \leq C$$
for some $C>0$ independent of $\A$ and $r\in (0,1)$.  This condition is satisfied e.g. sums of one-dimensional operators (possibly of different orders) satisfying 
\ref{A6}.
%on subspaces spanning $\rn$. 
\medskip

\noindent (c) By symmetry \ref{A5} and \ref{A7} it is clear that $\int_{\delta<|z|<1} \eta^{\A}(z) \, \nu_{\A}(dz) =0$. Hence we can also define $\I^\A$ in \eqref{n_local:term} using principal values and dropping the gradient (compensator) term. 
\medskip

 \noindent(d) Note that \ref{A3}--\ref{A7} give no restrictions on the tails of the L\'evy measures and the nonsingular part of the nonlocal operators. This possibly non-symmetric part could be the generator of any compound Poisson process.\ee

\medskip
\noindent(e) The fractional Laplacian $-(-\Delta)^{\frac\sigma2}$, where $\eta^\A=z$ and $\nu(dz)=\frac{c_{\alpha,N}}{|z|^{N+\sigma}}dz$, is a special case satisfying all assumptions \ref{A3}--\ref{A7}, see also section \ref{sec:fraclap}.
%The most general error bounds will depend on this $\sigma$ 
%Assumption \ref{A7} strengthens the symmetric structure of the nonlocal term $\I^{\A}$. It turns out to be important to improve the convergence rate compare to the most general case. \smallskip
%
%\noindent (b) Assumption \ref{A6} can be relaxed by assuming the order singularity varies with $\A$ as well. To be precise, it is possible replace the uniform upper bound of $\tilde{\nu}_{\A}(z)$ with the bound $\frac{C_{\A}}{|z|^{N+\sigma_{\A}}}$ for   \
%
% \noindent (b) By the assumptions it is clear, one can consider \bi $\frac{d\nu_{\A}}{dz}= \frac{C_{\A}}{|z|^{M_{\A}+ \sigma_{\A}}}$ for each $\A\in \mathcal{A}$ if $0< \sigma_{\A}\leq \sigma$, $M_{\A} \in \mathbb{N}$ and $0 < C_{\A} \leq C$. Near $z=0$, the nonlocal term would behave like fractional Laplacian $C_{\A} (-\Delta_{M_{\A}})^{\frac{\sigma_{\A}}{2}}[u]$  where the order of the singularity $\sigma_{\A}$ might be different for each $\A$. \smallskip
\end{rem}

A definition and general theory of viscosity solution for the nonlocal equations like \eqref{eqn:main} can be found e.g. in \cite{JK05,BI08}, but we do not need this generality here. In particular since there is no local diffusion, we could follow the simpler (comparison) arguments of \cite{CJ17}. Wellposedness and Lipschitz regularity for solutions of equation \eqref{eqn:main} are given in the next result. 
\begin{prop}\label{thm:viscosity_exist}
Assume \ref{A1}- \ref{A4}. 
\begin{itemize}
\item[(a)]  
If $u$ and $v$ are bounded upper semicontinuous viscosity subsolution and bounded lower semicontinuous supersolution of \eqref{eqn:main}, then 
$$u\leq v \quad in \quad \rn. $$
\item[(b)] There exists a unique viscosity solution $u\in C_b(\rn)$ of equation \eqref{eqn:main}.
\smallskip
\item[(c)] The viscosity solution $u$ of \eqref{eqn:main} is Lipschitz continuous,
$$\qquad\|u\|_0\leq \frac1\lambda\sup_{\A \in \mathcal{A}}\|f^{\A}\|_0,\qquad
%\text{and}\quad
\|Du\|_0\leq \frac1\lambda\sup_{\A \in \mathcal{A}}\big(\|Df^{\A}\|_0+\|Dc^{\A}\|_0\|u\|_0\big).$$
 
\end{itemize}

\end{prop} 

\begin{proof}
We refer to \cite{CJ17} Theorems 2.1, 2.3, and Corollary 2.3 for the proof (see also \cite{Ishii-R-21}) of parts (a), (b), and the first part of (c). The second estimate in (c) follows by the comparison principle in a standard way. 
\end{proof}

\subsection{Extra regularity for weakly non-degenerate equations}
A weakly non-degenerate version of \eqref{eqn:main} is
\begin{align}\label{eqn:main:wdeg:nonlocal}
\lambda u(x) + \sup_{\A \in \mathcal{A}} \left\{ f^{\A}(x) - \, \mathcal{I}^{\A} [u](x) \right\} = 0, 
\end{align}
where to simplify we have set $c^{\A}(x)\equiv \lambda>0$. We assume slightly more regularity of $f$ and weak degeneracy in the following sense:
\begin{Assumptions3}
\item \label{C1} \textbf{Weak-degeneracy:} There are $\A_0\in \mathcal{A}$, $c_{\A_0}>0$, and $K\geq 0$, such that %$c_{\A_0}>0$ and %such that the density from assumptions \ref{A6} satisfies 
\begin{align*}
    & (i)\quad \frac{d{\nu}_{\A_0}}{dz} \geq \frac{c_{\A_0}}{|z|^{N+\sigma}} & \mbox{for} \quad |z|<1, \\
    & (ii)\quad |\eta^{\A_0}(z) - \eta^{\A_0}(0) -  z|\leq K |z|^2 & \mbox{for} \quad |z|<1.
\end{align*}%\smallskip
% (C2) for part (ii)
\item \label{C3} There is $\beta > (\sigma-1)^+$ and $K>0$ such that %f^{\A}\in C^{1,\beta}(\rn)$ and 
$\|f^{\A}\|_{1,\beta} \leq K$ for every $\A\in \mathcal{A}$. 
\end{Assumptions3}

\begin{rem}
(a) Assumption \ref{C1} is a lower bound on the order of differentiability of $\I^{\A_0}$ and implies that it is elliptic/non-degenerate. The lower bounds behaves as $z\to0$ as the $\frac\sigma 2$-fractional Laplacian.
\smallskip

\noindent (b) weakly non-degenerate in \ref{C1} means that there is at least one $\alpha_0$ such that $\I^{\A_0}$ is non-degenerate. If $\I^{\A}$ is non-degenerate for all $\A$, with uniform bounds in \ref{C1}, then equation \eqref{eqn:main} is (uniformly/strongly) non-degenerate and have classical solutions. 
\end{rem} \ee

We prove our regularity results via an approximate problem where the L\'evy measure is truncated near origin: 
\begin{align}\label{eqn:aux:wdeg:nonlocal}
 \lambda u(x) + \sup_{\A \in \mathcal{A}} \left\{ f^{\A}(x) -  \, \mathcal{I}^{\A,r} [u](x) \right\} = 0 \qquad\text{in}\qquad \rn,
\end{align}   
where $\mathcal{I}^{\A,r}$ is defined by 
$$\mathcal{I}^{\A,r} \phi(x):= \int_{|z|>r} \Big( \phi(x+\eta^{\A}(z)) - \phi(x) \Big) \, \nu_{\A}(dz). $$
Note that $\I^{\alpha, r}$ is a bounded operator, well-defined for bounded functions, and then that viscosity solutions of equation \eqref{eqn:aux:wdeg:nonlocal} also will be pointwise/classical solutions. This problem is well-posed by Proposition \ref{thm:viscosity_exist}, and we have the following stability and approximation results:
\begin{lem} \label{lemma:auxeqn:wdeg:nonlocal}
Assume \ref{A1}-\ref{A4}, \ref{A6}, $u_{r}$ and $u$ are the unique bounded solutions of \eqref{eqn:aux:wdeg:nonlocal} and \eqref{eqn:main:wdeg:nonlocal}. Then there is a $C>0$ independent of $r$ such that
\begin{align*}
\|u_{r}\|_{0,1} \leq \frac{1}{\lambda} \sup_{\A\in\mathcal A}\|f^{\A}\|_{0,1} \qquad\text{and} \qquad \|u-u_{r}\|_0 \leq C \, r^{1-\frac{\sigma}{2}}.
\end{align*}
\end{lem} 
%ERJ
\begin{proof}
The first part follows from Proposition \ref{thm:viscosity_exist} (c).
 By a continuous dependence result,
 $$\|u-u_r\|_0 \leq K \sup_{\A \in \mathcal{A}}\Big(\int_{|z|<r} |z|^2 \, \nu_{\A}(dz)\Big)^{\frac{1}{2}}$$
 for some $K>0$ independent of $r$. Since $\int_{|z|<r} |z|^2 \, \nu_{\A}(dz)= C \, r^{2-\sigma}$ by \ref{A6}, the second part follows. The continuous dependence result is the stationary version of Theorem 4.1 in \cite{JK05} and can be proved in a similar way. We omit the proof here.
\end{proof}

We introduce a truncated fractional Laplacian,
\begin{align*}
 \Delta^{\sigma, r}[\phi](x) = \int_{|z|>r} \big(\phi(x + z) - \phi(x)\big) \, \frac{dz}{|z|^{N+\sigma}}.
\end{align*}

\begin{thm} \label{thm:reg:wdeg:nonlocal}
Assume \ref{A1}-\ref{A7}, \ref{C1}-\ref{C3}, and $u_{r}$ is the unique viscosity solution of \eqref{eqn:aux:wdeg:nonlocal}. Then for any $r>0$ there is a $K>0$ independent of $r$ such that 
\begin{align}\label{regbnd:aux:frac:nonocal}
\|\, \Delta^{\sigma, r} [u_{r}]\, \|_0 \leq \frac{K}{c_{\A_0}}. 
\end{align}

%\begin{itemize}
%\item[(a)] For any $r>0$ and $\A \in \mathcal{A}$, we have a constant $K$ independent of $r$ and $\A$ such that 
%\begin{align}\label{regbnd:aux:nonlocal}
%\|\, \J^{\A,s} [u_{r}]\, \|_0 \leq K . 
%\end{align}
%%
%%\item[(b)] Lower bound: \quad $\I^{\A,s}[u_r] \geq -\frac{K}{\delta}$ for any $s\geq r$, for some constant $K>0$ independent of $r$ and $s$. 
%%
%%\item[(c)] For any $s\geq r$ we have
%%\begin{align}\label{regbnd:aux:nonlocal}
%%\|\, \I^{\A,s} [u_{r}]\, \|_0 \leq K \quad \mbox{and} \quad  \|\, \J^{\sigma, s} [u_{r}]\, \|_0 \leq K
%%\end{align}
%%where the constant $K$ is independent of $r$ and $s$. 
%\item[(b)]  Furthermore, 
%\begin{align}\label{regbnd:aux:frac:nonocal}
%\|\, \Delta^{\sigma, r} [u_{r}]\, \|_0 \leq \frac{K}{c_{ \A_0}}. 
%\end{align}
%\end{itemize}
\end{thm}

\begin{proof}
Let us define the bounded auxiliary operator 
\begin{align*} 
    \J^{r}[\phi](x)= \int_{|z|>r} \big(\phi(x + \eta^{\A_0}(z)) - \phi(x)\big) \, \frac{c_{ \A_0}dz}{|z|^{N+\sigma}}.
\end{align*}
1)  \textit{A uniform bound on $w_r:=-\J^{r}[u_r]$.} \ Fix $x \in \rn$. By \eqref{eqn:aux:wdeg:nonlocal} and properties of suprema,  for any $\epsilon>0$ there exists $\bar{\A} \in \mathcal{A}$ such that 
\begin{align}\label{esti1:thm:reg}
\lambda \, u_{r}(x) + f^{\bar{\A}}(x) -  \, \mathcal{I}^{\bar{\A},r} u_{r}(x) \,  \geq - \epsilon, 
\end{align}
and (trivially) for any $y \in \rn$,
%Again by \eqref{eqn:aux:wdeg:nonlocal} and , for any $y \in \rn$ we have 
\begin{align} \label{esti2:thm:reg}
\lambda \, u_{r}(x+y) + f^{\bar{\A}}(x+y) - \, \mathcal{I}^{\bar{\A},r} u_{r}(x+y) \,  \leq 0. 
\end{align}
Take $y= \eta^{\A_0}(z)$, subtract equations  \eqref{esti1:thm:reg}  and \eqref{esti2:thm:reg}, multiply by $\frac{c_{\A_0}}{|z|^{N+\sigma}}$, and integrate over $|z|>r$. The result is then  
\begin{align*}
&\lambda \,\big(-\J^{r} [u_{r}](x)\big) - \J^{r} [f^{\bar{\A}}](x) -\J^{r}\big[-\I^{\bar{\A}, r}[u_{r}]\big](x) \geq -\epsilon. 
\end{align*}
This inequality holds for $\bA$ and then also holds for the supremum over all $\A\in\mathcal A$. Since $\epsilon>0$ and $x\in\mathbb R^N$ are arbitrary, $\J^{r}$ and $\I^{\A, r}$ are linear operators, and by Fubini $\J^{r}\big[\I^{\bA, r}[u_{r}]\big] =  \I^{\bA, r}\big[\J^{r}[u_{r}]\big]$, by the definition of $w_r$ we have
\begin{align}\label{comparison:auxeqn:nonlocal}
\lambda w_{r}(x) +  \sup_{\A\in \mathcal{A}}\big\{-\J^{r} [f^{\A}](x) -  \I^{\A, r}[w_{r}](x) \big\} \geq 0 \qquad \text{in}\qquad \mathbb R^N.
\end{align}
By assumption \ref{C3}, $C:=\sup_{\A\in \mathcal{A}}\|\J^{r} [f^{\A}]\|_0 <\infty$, so $-\frac{C}{\lambda}$ is a subsolution  of \eqref{comparison:auxeqn:nonlocal}.\footnote{Replace $\geq$ by $=$ in \eqref{comparison:auxeqn:nonlocal}.}  Then by comparison, 
Proposition \ref{thm:viscosity_exist}~(a),\footnote{Equation \eqref{comparison:auxeqn:nonlocal} (replace $\geq$ by $=$) is of same form as in \eqref{eqn:main}.} 
 \begin{align} \label{lowerbd:eqn:nonlocal}
 -\J^{r}[u_r]=w_{r} \geq - \frac{C}{\lambda}\qquad \text{in}\qquad \rn. 
 \end{align}
 
To get a lower bound on $\J^{r}[u_{r}]$, we use the upper bound and weak degeneracy:  $\tilde{\nu}_{\A_0}(z) - \frac{c_{\A_0}}{|z|^{N+\sigma}} \geq 0$ for $|z|<1$. Let $y = \eta^{\A_0}(z)$, subtract  \eqref{esti2:thm:reg} and \eqref{esti1:thm:reg}, multiply by $(\tilde{\nu}_{\A_{0}}(z) - \frac{c_{ \A_0}}{|z|^{N+\sigma}})$, and integrate over $r<|z|<1$. The result is 
\begin{align*}
& \lambda \,\big(- (\I^{\A_0,r}_1 - \J^{r}_1 \, ) [u_{r}](x)\big) - (\I^{\A_0,r}_1 - \J^{r}_1 \, ) [f^{\bar{\A}}](x)  \\
& \hspace*{4cm} -   \I^{\bar{\A}, r}\Big[-(\I^{\A_0,r}_1 - \J^{r}_1 \, ) [u_{r}]\Big](x) \geq -\epsilon . 
\end{align*}
where $\mathcal{J}^{r}_1[\phi](x) = \int_{r<|z|<1} \big(\phi(x + \eta^{\A_0}(z)) - \phi(x)\big) \, \frac{c_{\A_0}dz}{|z|^{N+\sigma}}$.
Then arguing as for the upper bound we have 
\begin{align}\label{prf_err_M3}
-(\I^{\A_0,r}_1 - \J^{r}_1 \, ) [u_{r}]\geq -\frac{C}{\lambda}\qquad \text{in}\qquad \rn.
\end{align}
The above estimate implies $-\J^{r}_1 [u_{r}](x) \leq \frac{C}{\lambda} + \sup_{\A \in \mathcal{A}}\big\{- \I^{\A_0,r}_1 [u_{r}](x)\big\},$ and therefore since $u_r$ solves \eqref{eqn:aux:wdeg:nonlocal}, that
\begin{align}
%\begin{split}
&- \J^{r}_1 [u_{r}](x) \nonumber \\
&\leq  \frac{C}{\lambda} + \sup_{\A \in \mathcal{A}} \left\{ - \I^{\A,r}[u_r] + f^{\A}(x) \right\}  + \lambda u_r(x) \nonumber \\
&\quad+ \sup_{\A \in \mathcal{A}}\|f^{\A}\|_0 + \lambda \|u_r\|_0  + \sup_{\A \in \mathcal{A}}\Big|\int_{|z|>1} \big(u_r(x+ \eta^{\A}(z)) -u_r(x) \big) \nu_{\A}(dz) \Big|\nonumber \\
&\leq  \frac{C}{\lambda} + 0 + \sup_{\A \in \mathcal{A}}\|f^{\A}\|_0 + \Big(\lambda + 2 \sup_{\A\in \mathcal{A}} \int_{|z|>1} \nu_{\A}(dz) \Big) \|u_r\|_0  .\label{prf_err_M4}
%\end{split}
\end{align}
Let $\J^{r}=\J_1^{r}+\J^{1,r}$ where $\J^{1,r}=\int_{|z|>1}(\cdots)\frac{c_{\A_0}dz}{|z|^{N+\sigma}}$. By \ref{A2}, \ref{A4}, and Lemma \ref{lemma:auxeqn:wdeg:nonlocal}, both the right hand side of \eqref{prf_err_M4} and $\J^{1,r}[u_r]$ are bounded, and hence
\begin{align}\label{uperbd:eqn:nonlocal}
- \J^{r}[u_r] \leq C\qquad \text{in}\qquad \rn,
\end{align}
for some constant $C>0$ independent of $r$. By \eqref{lowerbd:eqn:nonlocal} and \eqref{uperbd:eqn:nonlocal} we conclude that $|w_r|=|\J^{r}[u_r]|\leq C_1$ for some other $C_1>0$  independent of $r$.\medskip
 
\noindent 2) \textit{The bound on $\Delta^{\sigma, r} [u_{r}]$.} \ Since $c_{ \A_0} >0$ by \ref{C1}, from step 1) it follows that 
\begin{align*}
I:=\Big|\int_{|z|>r} \big(u_r(x+\eta^{\A_0}(z)) - u_r(x)\big) \, \frac{dz}{|z|^{N+\sigma}}\Big| \leq \frac{C_1}{c_{ \A_0}}. 
\end{align*}
From this estimate, the bound $\|u_r\|_{0,1}\leq K$, and \ref{C1}$(ii)$ and \ref{A3} (implying $\eta^\alpha(0)=0$), we see that
\begin{align*}
|\Delta^{\sigma, r}[u_r](x)| &  \leq I+  \int_{|z|>r} \big|u_r(x+\eta^{\A_0}(z)) - u_r(x+z)\big| \, \frac{dz}{|z|^{N+\sigma}} \\
& \leq  \frac{C_1}{c_{ \A_0}} + \|Du_r\|_0 \int_{r<|z|<1} |z|^2 \, \frac{dz}{|z|^{N+\sigma}}  + 2\|u_r\|_0 \int_{|z|>1} \frac{dz}{|z|^{N+\sigma}}. 
\end{align*}
 The right hand side is uniformly bounded so the proof is complete.
\end{proof}

Sending $r\to0$ in the above result, we get a key result for this paper.
\begin{cor}\label{cor:fracbd:soln:nonlocal}
Assume \ref{A1}-\ref{A7}, \ref{C1}-\ref{C3}, and  $u$ it the unique viscosity solution of \eqref{eqn:main:wdeg:nonlocal}. Then $(-\Delta)^{\frac{\sigma}{2}}[u] \in L^{\infty}(\rn)$.
\end{cor}

\begin{proof}
Note that since $u$ is bounded, $(-\Delta)^{\frac{\sigma}{2}}[u] $ defines a distribution by 
$$( (-\Delta)^{\frac{\sigma}{2}}[u], \phi ) = \int_{\rn} u(x) \, (-\Delta)^{\frac{\sigma}{2}}[\phi](x)  \, dx \quad \mbox{for any} \quad \phi \in C^{\infty}_c(\rn).  $$
To complete the proof we must show that this distribution can be represented by a function in $L^\infty(\rn)$. Let $u_r$ be the bounded solution of \eqref{eqn:aux:wdeg:nonlocal}, and note that
\begin{align}\label{cor:esti1}
& \Big|\int_{\rn} u(x) \, (-\Delta)^{\frac{\sigma}{2}}[\phi](x)  \, dx - \int_{\rn} u_r(x) (-\Delta^{\sigma,r}[\phi](x))\, dx \Big| \notag \\
 &\leq  \Big| \int_{\rn}(u-u_r)(x) (-\Delta)^{\frac{\sigma}{2}}[\phi](x) \, dx \Big| + \|u_r\|_0 I,
 %\notag \\ 
%& \hspace*{2cm} + \|u_r\|_0 \Big| \int_{\rn} \big((-\Delta)^{\frac{\sigma}{2}}-
%(-\Delta^{\sigma,r}))[\phi]  \big)(x)\, dx \Big| . 
\end{align} 
where $(-\Delta)^{\frac{\sigma}{2}}[\phi] \in L^{1}(\rn)$\footnote{A Taylor expansion shows that $\|(-\Delta)^{\frac{\sigma}{2}}[\phi]\|_{L^1}\leq c \|\phi\|_{W^{2,1}}$, and $\|\phi\|_{W^{2,1}}<\infty$ for $\phi\in C^\infty_c$.} and by Taylor,
\begin{align*}
   I= &\ \int_{\rn}\Big| \big(-\Delta^{\sigma,r}[\phi] - (-\Delta)^{\frac{\sigma}{2}}[\phi] \big)(x) \Big|\, dx  \\
= &\ \int_{\rn} \Big| \int_{|z|<r} \big(\phi(x+z)-\phi(x) -z\cdot \grad\phi(x)\big)\frac{dz}{|z|^{N+\sigma}}\Big| dx \\
 \leq &\  \|D^2 \phi\|_{L^{1}(\rn)}  \int_{|z|<r} |z|^2 \frac{dz}{|z|^{N+\sigma}} \, \leq C \|D^2 \phi\|_{L^{1}(\rn)} r^{2-\sigma}.
\end{align*}
 By Lemma \ref{lemma:auxeqn:wdeg:nonlocal},  $\|u_r\|_0 $ is bounded independently of $r$ and $u_r\to u$ in $L^\infty$, hence since $\Delta^{\sigma,r}$ is self-adjoint, it follows from \eqref{cor:esti1} that
\begin{align}\label{cor:weak:eq}
\int_{\rn} u(x) \, (-\Delta)^{\frac{\sigma}{2}}[\phi](x) \, dx &= \lim_{r\rightarrow 0} \int_{\rn} u_r(x) (-\Delta^{\sigma,r}[\phi])(x)\, dx  \notag \\
& = \lim_{r\rightarrow 0} \int_{\rn} (-\Delta^{\sigma,r}[u_r])(x) \, \phi(x)\, dx.
\end{align}
By Theorem \ref{thm:reg:wdeg:nonlocal}, $\|\Delta^{\sigma,r}[u_r] \|_0 \leq K$ for some $K>0$ independent of $r$. By weak star compactness (Alaoglou/Helly) there is an $f \in L^{\infty}(\rn)$ %with $\|f\|_0 \leq K$ 
and a subsequence $\{r_n\}_n$  such that $r_n\to0$
and $(-\Delta^{\sigma,r_n}[u_{r_n}]) \stackrel{*}{\rightharpoonup} f$ in $L^\infty$.
Passing to the limit in \eqref{cor:weak:eq}, 
\begin{align*}
\int_{\rn} u(x) \, (-\Delta)^{\frac{\sigma}{2}}[\phi](x) \, dx  = \lim_{n\rightarrow \infty} \int_{\rn} (-\Delta^{\sigma,r_n}[u_{r_n}])(x) \, \phi(x)\, dx = \int_{\rn} f(x) \, \phi(x) \, dx. 
\end{align*}
The proof is complete. 
%We also note that the arguments  follows with no changes for any $\phi \in \mathcal{S}$, the space of all Schwartz class functions. Hence the distribution $ \J [u] \in L^{\infty}(\rn)$ and 
%$$\|\, \J [u]\, \|_0 \leq K.$$ 
\end{proof}  
We immediately observe an improvement of regularity for the viscosity solution of \eqref{eqn:main:wdeg:nonlocal} in the case that $\sigma>1$ (compare with Proposition \ref{thm:viscosity_exist}). 
%Note that we already have that the viscosity solutions are Lipschitz continuous for any order $\sigma\in (0,2)$.   

\begin{thm}\label{thm:regularity:vis_soln:wdeg:nonlocal}
Assume $\sigma>1$, \ref{A1}-\ref{A7}, \ref{C1}-\ref{C3}, and $u$ is the unique viscosity solution of \eqref{eqn:main:wdeg:nonlocal}. Then $u \in C^{1, \sigma -1}(\rn)$  and 
$$\|u\|_{1, \sigma -1} \leq K \big(\|u\|_0 + \|\, (-\Delta)^{\frac{\sigma}{2}} [u]\, \|_0\big).$$  
\end{thm}

\begin{proof}
By the Corollary \ref{cor:fracbd:soln:nonlocal}, $(-\Delta)^{\frac{\sigma}{2}} [u] \in L^{\infty}(\rn)$, and from the definition of viscosity solution $u \in L^{\infty}(\rn)$. Therefore the result follows from Theorem 1.1(a) of the article \cite{Oton2016} by Ros-Oton and Serra. 
\end{proof}

\begin{rem}
%(a) For $\sigma>1$, $C^{1, \sigma -1}(\rn)$ regularity is borderline but still not enough to have point-wise classical solutions. IS THIS TRUE? WHAT DOES ROS-OTON/SERA SAY?? I thought this regularity was enough to have the operator defined in any point...!? But may it is not continuous?? \smallskip

\noindent  When $\sigma< 1$ we get no improvement in regularity from Lipschitz (Proposition \ref{thm:viscosity_exist}(c)). But here Lipschitz regularity is sufficient for solutions to be point-wise classical solutions of  \eqref{eqn:main:wdeg:nonlocal}.
\end{rem}

\section{Diffusion corrected difference-quadrature scheme} \label{monotone scheme}

In this section we construct  monotone discretizations for equation \eqref{eqn:main} (and \eqref{eqn:main:wdeg:nonlocal}),
%for general nonlocal terms for which the 
and give precise results on their convergence rates. There are two main steps to construct the schemes: (i) approximate the singular part of the nonlocal operator by a local diffusion, and (ii) discretize the resulting equations using semi-Lagrangian type of difference quadrature schemes. 

%The non-local operators $\I^{\A}$ have %very diferent part in nature 
%two different parts, one is the singular part near origin and other one is the bounded part and we discretize them  separately (to get optimal convergence rate). 
By symmetry \ref{A5} and \ref{A7}, $\left(\int_{\delta<|z|<1} \eta^{\A}(z) \, \nu_{\A}(dz) \right)\cdot \nabla \phi(x)=0$. For $\delta\in(0,1)$, we then write the nonlocal operator $\I^{\A}$ as
%in the following way %form for sufficiently smooth functions $\phi$ and : 
\begin{align}\label{eqn:neq_form1}
\I^{\A}[\phi](x) & = \left( \int_{|z|<\delta} + \int_{|z|>\delta}\right) \Big(\phi(x+\eta^{\A}(z)) - \phi(t, x) -\eta^{\A}(z)\cdot \nabla \phi(x) \Big)\, \nu_{\A}(dz) \notag \\
& = \int_{|z|<\delta}\Big(\phi(x+\eta^{\A}(z)) - \phi(t, x) -\eta^{\A}(z)\cdot \nabla \phi(x) \Big)\, \nu_{\A}(dz) \notag \\
& \hspace*{4cm} +\int_{|z|>\delta} \Big(\phi(x+\eta^{\A}(z)) - \phi(t, x) \Big)\, \nu_{\A}(dz)   \notag \\
& := \I^{\A}_{\delta} [\phi](x) + \I^{\A,\delta}[\phi](x) . 
\end{align}
The $\delta$ will be chosen later. We say that $\I^{\A}_\delta$ is the singular part\footnote{When $\nu$ has a singularity at the origin, this is a singular integral operator. If the singularity is strong enough, the operator will be a fractional differential operator of positive order.} of $\I^{\A}$, while $\I^{\A,\delta}$ is always a bounded operator.

\subsection{Approximation of the singular part of the nonlocal operator. }
 The simplest (but not very accurate) discretization of $\I^{\A}_\delta[\phi]$ is to replace it by $0$. Better approximations %the small jumps by the 
 can be obtained using local diffusion terms \cite{CT:Book,JKL08}. This corresponds to approximating the small jumps in the SDE \eqref{SDE} by an appropriate Brownian motion \cite{AR01}. We define 
 $$a^{\A}_{\delta} = \frac{1}{2}\int_{|z|<\delta} \eta^{\A}(z)\eta^{\A}(z)^T \, \nu_{\A}(dz)\qquad\text{and}\qquad \mathcal{L}_\delta^{\A}[\phi](x) :=tr[a^{\A}_{\delta} D^2 \phi],$$
  where $a^{\A}_{\delta}$ is a constant non-negative %positive semi-definite
  matrix and $\phi\in C_b^2(\rn)$. We approximate equation \eqref{eqn:main} %(as well as equation \eqref{eqn:main:wdeg:nonlocal} ) 
  by replacing $\I_{\delta}^{\A}[\phi]$ with $\mathcal{L}_\delta^{\A}[\phi](x)$:
\begin{align}\label{eqn:main:apprx}
\sup_{\A \in\mathcal{A}} \Big\{ f^{\A}(x)   + c^{\A}(x)u(x) - \mathcal{L}_\delta^{\A}[\phi](x) - \I^{\A,\delta}[u](x) \Big\} =0 \quad \mbox{in} \quad \rn. 
\end{align}
 
\begin{lem}\label{lem:err_local_dif_sym_odd}
Assume \ref{A1}-\ref{A7} and $\delta\in(0,1)$.  Then there are $C,K>0$ independent of $\delta,\alpha,\phi$ such that
\begin{align}\label{err_local_dif_sym}
&(i)\quad |\I^{\A}_{\delta}[\phi]- \mathcal{L}_\delta^{\A}[\phi]
%tr[a^{\A}_{\delta}D^2 \phi] 
| \leq C\delta^{4-\sigma} \|D^4\phi\|_0,
 \\[0.2cm]
 &(ii) \quad  |a^{\A}_{\delta}| \leq \int_{|z|\leq \delta}
  |\eta^{\A}(z)|^2 \, \nu_{\A}(dz) \leq K \delta^{2-\sigma} \label{a-bnd}.
\end{align} 
\end{lem}

\begin{proof}
By Taylor's expansion theorem  and  smooth $\phi$,
\begin{align*}
& \int_{|z|<\delta}\Big(\phi(x+\eta^{\A}(z))+ \phi(x) - \eta^{\A}(z)\cdot \grad \phi \Big) \, \nu_{\A}(dz) \\
& =\int_{|z|<\delta} \Big( \eta^{\A}(z)\cdot D^2\phi(x)\cdot \eta^{\A}(z)^T + \sum_{|\beta|=3 } \frac{1}{\beta!} [\eta^{\A}(z)]^{\beta} D^{\beta}\phi(x) \Big) \, \nu_{\A}(dz) + Err_{\delta},
\end{align*}
where $Err_{\delta}= \frac{|\beta|}{\beta!}\sum_{|\beta|=4}\big[\int_{|z|<\delta} \int_0^1(1-s)^{|\beta-1|} [\eta^{\A}(z)]^{\beta} \, D^{\beta}\phi(x+sz) \, ds\, \nu_{\A}(dz)\big].$  
By the assumptions \ref{A5} and \ref{A7} and then by \ref{A6} we have
$$\sum_{|\beta|=3 } \int_{|z|\leq \delta} \, [\eta^{\A}(z)]^{\beta} D^{\beta}\phi(x) \,\nu_{\A}(dz) =0 \quad \text{and} \quad |Err_{\delta}|\leq C \delta^{4-\sigma} \|D^4\phi\|_0.$$ 
That proves part $(i)$. Part $(ii)$ follows by \ref{A3} and \ref{A4}. 
\end{proof}

% We further note that with the assumption \ref{A3} we have $|\tilde{b}^{\A,\delta}|\leq \int_{|z|>\delta} |\eta^{\A}(z)| \, \nu(dz) \leq C \, \Gamma(\delta, \sigma),  $ where the function $\Gamma(\sigma, \delta)$ is defined as 
%\begin{align*} %\label{delta_condn}
%\Gamma(\sigma, \delta) =
%\left\{
%       \begin{array}{ll}
%       \delta^{1-\sigma} \quad & \text{when} \quad \sigma > 1, \\
%        - \log \delta \quad & \text{when} \quad \sigma = 1,\\
%       1 \quad & \text{when} \quad \sigma <1.
%       \end{array}
%\right.
%\end{align*} 

\subsection{Consistent monotone discretization of the approximate equation}\label{subsec:discr} 
%We discretize the local operator of equation \eqref{eqn:main:apprx}  in the same line of approach as in  \cite{JKL08} (also see,\cite{BJK2, BJK1}).
%We assume either \ref{A6} is true or $\displaystyle \sup_{\A\in \mathcal{A}, |z|<1} |D^2 \eta^{\A}(z)| \leq K$.  
%$$\Delta_{h,l}u(x)= \frac{u(x+hl)+u(x-hl)-2u(x)}{h^2}. $$ 
We now approximate the local and nonlocal part of equation \eqref{eqn:main:apprx} separately.\\

\noindent \textit{(i) Discretization of the local term:}  
Since $a^{\A}_{\delta}$ is symmetric and nonnegative ($\xi^T a^{\A}_{\delta}\xi = \int_{|z|<\delta} (\eta^{\A}(z)\cdot \xi)^2 \, \nu_{\A}(dz) \geq 0 $), it has a square root with columns $(\sqrt{a^{\A}_{\delta}})_{i}$. We then introduce
%approximate the term $tr[a^{\A}_{\delta} D^2 \phi]$ by 
the semi Lagrangian (SL)  approximation (inspired by \cite{Camilli-Falcone, Deb_Jak:13})
\begin{align}\label{SL_approx_term}
\mathcal{L}_\delta^{\A}[\phi]&=tr[a^{\A}_{\delta} D^2 \phi] \nonumber\\
& \approx \sum_{i=1}^N \frac{\phi(x + k (\sqrt{a^{\A}_{\delta}})_{i}) + \phi(x - k (\sqrt{a^{\A}_{\delta}})_{i}) - 2\phi(x)}{2\,k^2} \equiv \mathcal{D}^{\A}_{\delta,k}[\phi](x). 
\end{align} 
This approximation is monotone by construction, and by Taylor expansions, 
%the local truncation error 
%(c.f. Lemma 4.5 in \cite{EJT18b} for the proof)  
\begin{align}\label{SL-bd}
|\mathcal{L}_\delta^{\A}[\phi] - \mathcal{D}^{\A}_{\delta,k}[\phi] | \leq K |a^{\A}_{\delta}|^2 k^2 \|D^4 \phi\|_0 \leq K\delta^{2(2-\sigma)} k^2 \|D^4 \phi\|_0.
\end{align}
Since $x_{\bj} \pm k (\sqrt{a^{\A}_{\delta}})_{i}$
may not be on the grid, we interpolate to get a full discretization. To preserve monotonicity, we use linear/multilinear interpolation $i_h(\phi)(x)=\sum_{\bj \in \zn} \phi(x_j) \omega_{\bj}(x)$ where the basis functions $\omega_{\bj}\geq 0$ and $\sum_{\bj\in \zn} \omega_{\bj}=1$. Let 
%Finally the monotone discretization of the 2nd order term $tr[a^{\A}_{\delta} D^2 \phi]$ is given by
\begin{align}\label{local_epprox_term}
\mathcal{L}^{\A}_{\delta,k,h}[\phi](x) = \sum_{i=1}^N \frac{i_h\big[\phi(x + k (\sqrt{a^{\A}_{\delta}})_{i})\big] + i_h\big[\phi(x - k (\sqrt{a^{\A}_{\delta}})_{i})\big] - 2\phi(x)}{2 k^2}. 
\end{align}
By the property of multilinear interpolation, this approximation is monotone with %truncation error 
\begin{align}\label{SL_bd_interpolant}
|\mathcal{L}^{\A}_{\delta,k,h}[\phi] - \mathcal{D}^{\A}_{\delta,k}[\phi]| \leq C \frac{h^2}{k^2} \|D^2 \phi\|_0 . 
\end{align}  
By \eqref{SL-bd} and \eqref{SL_bd_interpolant} we have a truncation error bound for the local approximate term. 

\begin{lem} \label{lem:local_trunc1}
Assume \ref{A3}-\ref{A7}. Then there is $K>0$ independent of $h,\delta,\A,\phi$ such that \begin{align}\label{local_trunc1}
 \big| \mathcal{L}^{\A}_{\delta,k,h} [\phi](x) - \mathcal{L}_\delta^{\A} [\phi](x)  \big| \leq K \Big( \delta^{2(2-\sigma)} k^2 \|D^4\phi\|_0 + \frac{h^2}{k^2} \|D^2 \phi\|_0\Big).  
\end{align} 
%for some $K>0$ independent of $h,k,$ and $\delta$. 
\end{lem} 

\smallskip
\noindent \textit{(ii) Discretization of the nonlocal term: } We follow  \cite[Section~3]{BCJ1} and approximate $\I^{\A, \delta}$ by the quadrature
 \begin{align}\label{eq:discrete_nonlocal}
   \I_h^{\A, \delta}[\phi] = 
%& \int_{\R^M\backslash\{0\}}
%\sum_{\bj\in\zn}\big(\phi(x+x_{\bj})-\phi(x)\big)
%\omega_{\bj}(\eta^{\A}(z); h)\nu_{\A}(dz)= 
\sum_{\bj \in \Z^N} \big(\phi(x+x_{\bj})-\phi(x)\big)
   \kappa_{h,\bj}^{\A,\delta}; \quad  \kappa_{h,\bj}^{\A,\delta} =\textstyle \int_{|z|>\delta} \omega_{\bj}(\eta^{\A}(z); h)\nu_{\A}(dz),
 \end{align}
where $\{\omega_{\bj}\}_\bj$ is the basis  for multilinear interpolation defined above.              
Since $\omega_{\bj}\geq 0$,  $\kappa_{h,\bj}^{\A,\delta}\geq 0$, and the approximation $\I^{\A, \delta}_h$ is monotone. A Taylor expansion gives an estimate on the local truncation error, c.f. \cite{BCJ1}: 
\begin{lem}\label{lem:trunc_err2}
Assume \ref{A3}-\ref{A4} and \ref{A6}. Then  there is $K>0$ independent of $h,\delta,\A,\phi$ such that
\begin{align}\label{nonl_trunc1}
\big| \I^{\A,\delta}[\phi](x) - \I_h^{\A,\delta}[\phi](x)\big|  \leq K \frac{h^2}{\delta^{\sigma}} \|D^2 \phi\|_0.
\end{align} 

\end{lem}
 
 \smallskip
\noindent \textit{(iii) Discretization of the nonlocal equation \eqref{eqn:main}:} 
%Finally the monotone numerical scheme for \eqref{eqn:main} can be written as 
\begin{align}\label{approx:eqn1}
 \sup_{\A \in \mathcal{A}} \Big\{ f^{\A}(x) + c^{\A}(x)u(x)- \mathcal{L}^{\A}_{\delta,k,h} [u](x) - \mathcal{I}_h^{\A, \delta}[u](x) \big\} =0 \quad \text{in} \quad \rn,
\end{align}  
or in weakly non-degenerate case \eqref{eqn:main:wdeg:nonlocal} where $c^{\A}(x)= \lambda$, 
%the fully discrete monotone numerical approximation of \eqref{eqn:main:wdeg:nonlocal} is given by
\begin{align}\label{eqn:approx:wdeg:nonlocal}
\lambda \, v(x) + \sup_{\A \in \mathcal{A}} \left\{ f^{\A}(x) - \mathcal{L}^{\A}_{\delta,k,h} [v](x) - \mathcal{I}_h^{\A, \delta}[v](x) \right\} =0 \quad \text{in} \quad \rn. 
 \end{align} 
% where the local approximation is of the form $\mathcal{L}_\delta^{\A}_h [v](x)=\mathcal{L}^{\A}_{\delta,k,h} [\phi]$. 

\subsection{Properties and convergence analysis for the schemes}

We state wellposedness, comparison, $L^\infty$-stability, and $L^\infty$-convergence results for the schemes in different settings. 
\begin{thm}[wellposedness, stability]\label{thm:epprx_exist1}
Assume \ref{A1}-\ref{A4}.
\begin{itemize}
\item[(a)] There exists a unique %continuous bounded 
solution $u_h\in C_b(\rn)$ of \eqref{approx:eqn1}. \smallskip
\item[(b)] If $u_h,v_h \in C_b(\rn)$ are sub and supersolutions of \eqref{approx:eqn1}, then $u_h\leq v_h$.\smallskip
\item[(c)]  If $u_h$ is the unique solution of \eqref{approx:eqn1}, then $|u_h|_0 \leq C \sup_{\A \in \mathcal{A}} |f^{\A}|_0.$
\end{itemize} 
\end{thm}
\begin{proof}
Part~(a)  can be proved using Banach fixed point arguments, we refer to \cite[Lemma 3.1]{BJK3} for details. Part (b) is a consequence of the scheme having positive coefficients. Finally, part (c) follows from (b) by taking $\pm \frac{1}{\lambda}\sup_{\A \in \mathcal{A}} |f^{\A}|_0$ as super and sub-solution of the scheme \eqref{approx:eqn1} respectively.  
\end{proof}

If the solutions of \eqref{eqn:main} are very smooth ($C_b^4$), then we get the best possible convergence rate for our scheme -- what some would call the accuracy of the method: 
\begin{prop}[Smooth solutions] \label{prop:accuracy-scheme1}
Assume \ref{A4}--\ref{A7}, $\sigma\in(0,2)$, $u\in C^4_b(\rn)$ solves \eqref{eqn:main}, and  $u_h$ solves \eqref{approx:eqn1} with $k=O(h^{\frac\sigma4})$ and $\delta= O(h^{\frac{1}{2}})$. Then
there is $C>0$ such that
\begin{align*}
    |u-u_h|\leq C h^{2-\frac{\sigma}{2}}.
\end{align*}
\end{prop}
This rate is always better than $1$, and approaches $1$ as $\sigma\to 2^-$. We will not discuss assumptions to have so smooth solutions, but below we will give results that holds for the solutions that exist under the assumptions of this paper.
\begin{proof}
By equation \eqref{eqn:main} and the errors bounds \eqref{err_local_dif_sym}, \eqref{local_trunc1}, \eqref{nonl_trunc1},  for any $\A\in \mathcal{A}$, 
\begin{align*}
    &f^{\A}(x) + c^{\A}(x)u(x)- \mathcal{L}^{\A}_{\delta,k,h} [u](x) - \mathcal{I}_h^{\A, \delta}[u](x)  \leq \mathcal{I}^{\A}[u](x) - \mathcal{L}^{\A}_{\delta,k,h} [u](x) - \mathcal{I}_h^{\A, \delta}[u](x) \\
    & \leq C \Big( \delta^{4-\sigma} \|D^4u\|_0 + \frac{h^2}{k^2}\|D^2u\|_0 + \delta^{2(2-\sigma)} k^2 \|D^4u\|_0 + \frac{h^2}{\delta^{\sigma}} \|D^2u\|_0 \Big) := B_{h, \delta}.
\end{align*}
This implies $u(x) - \frac{B_{h, \delta}}{\lambda}$ is a subsolution of \eqref{approx:eqn1}, and by Theorem \ref{thm:epprx_exist1} (b) that 
\begin{align*}
    u-u_h\leq \frac{B_{h, \delta}}{\lambda}.
\end{align*}
Again by \eqref{eqn:main}, the definition of the sup, and the errors bounds, for any $x\in \rn$ and $\epsilon>0$, there is a $\A_{\epsilon} \in \mathcal{A}$ such that 
\begin{align*}
    & f^{\A_{\epsilon}}(x) + c^{\A_{\epsilon}}(x)u(x) - \mathcal{L}^{\A_{\epsilon},\delta}_{k,h} [u](x) - \mathcal{I}_h^{\A_{\epsilon},\delta}[u](x) \\ 
    & \geq -\epsilon + \mathcal{I}^{\A_{\epsilon}}[u](x) - \mathcal{L}^{\A_{\epsilon},\delta}_{k,h} [u](x) - \mathcal{I}_h^{\A_{\epsilon},\delta}[u](x) \geq -\epsilon - B_{h,\delta}.
\end{align*}
 Let $\tilde{u} = u + \frac{B_{h,\delta}}{\lambda}$, and note that
\begin{align*}
   \sup_{\A \in \mathcal{A}} \Big\{ f^{\A}(x) + c^{\A}(x) \tilde{u} (x) - \mathcal{L}^{\A}_{\delta,k,h} [\tilde{u}](x) - \mathcal{I}_h^{\A, \delta}[\tilde{u}](x) \big\} \geq -\epsilon.
\end{align*}
Since $\epsilon$ and $x$ are arbitrary, $\tilde{u}$ is a supersolution of \eqref{approx:eqn1}, and then $u_h-u \leq \frac{B_{h,\delta}}{\lambda}$ by Theorem \ref{thm:epprx_exist1} (b). Since $u\in C^4_b(\rn)$, we have shown that
\begin{align*}
    |u-u_h| \leq \frac{C}{\lambda} \Big( \delta^{4-\sigma}  + \frac{h^2}{k^2} + \delta^{2(2-\sigma)}k^2  + \frac{h^2}{\delta^{\sigma}} \Big).
\end{align*}
We conclude by taking the optimal choices $k^2= O(\frac{h}{\delta^{2-\sigma}})$ and then  $\delta = O(h^{\frac{1}{2}})$.
\end{proof}
% ERJ

The next two results form the main contribution of this paper along with the result of section \ref{sec:fraclap}. These results give very precise rates of convergence for our monotone numerical approximations in cases of strongly and weakly non-degenerate equations respectively. 
 Note that in these results the solutions $u$ of \eqref{eqn:main} and \eqref{eqn:main:wdeg:nonlocal} will not be smooth. The proofs of these results are given in Section~ \ref{sec_prf_gen}.

\begin{thm}[Strongly degenerate equations]\label{thm:estimate_sym_main}
Assume $\sigma\in(0,2)$, $h\in(0,1)$, \ref{A1}-\ref{A7}, $u$ and $u_h$ are solutions of  \eqref{eqn:main} and \eqref{approx:eqn1} for %$k^2=\frac{h^2}{\delta^{2-\frac{\sigma}{2}}}$ and 
$k=O(h^{\frac{2\sigma}{4+\sigma}})$
and  $\delta= O(h^{\frac{4}{4+\sigma}})$. 
Then 
there is a $C>0$ such that 
%for any $0<\sigma<2$ we have
\begin{align}\label{estimate_sym_nodrift}
|u- u_{h}| \leq C \, h^{\frac{4-\sigma}{4+\sigma}}. 
\end{align}

\end{thm}

\begin{rem}

\noindent (a) \,   
%For $\sigma>1$, t
The rate $\frac{4-\sigma}{4+\sigma}$ is decreasing in $\sigma$. It equals $\frac35$ at $\sigma=1$, approaches $1$ as $\sigma\to0^+$, and $\frac{1}{3}$ as $\sigma\to 2^-$. \smallskip

\noindent (b) \, The "CFL" conditions $k=O(h^{\frac{2\sigma}{4+\sigma}})$ and $\delta= O(h^{\frac{4}{4+\sigma}})$ imply that $\frac hk\to 0$ and $\frac h\delta\to 0$ as $h\to 0$. %This means that the optimal rate is attained for long time steps $k$. \ee
\smallskip

\noindent (c) \,   Conditions \ref{A5} and \ref{A7} are symmetry assumptions on the singular part of $\I^{\A}$ which lead to best possible rates. We refer to Section \ref{sec:exten} for extensions to nonsymmetric nonlocal operators and the corresponding (slightly) lower rates.  
%this is due to the fact that as $\sigma$ close to $0$ the effect of nonlocal term in the equation become less and less   \\ 
\end{rem}

In the weakly non-degenerate case we get an improvement in the rate due to the better regularity of solutions both for the equation and the numerical scheme:  

\begin{thm}[weakly non-degenerate equations]\label{thm:er_bound:wdeg}
Assume $\sigma\in(0,2)$, $h\in(0,1)$, \ref{A1}-\ref{A7}, \ref{C1}-\ref{C3}, $u$ and $u_h$  are the solutions of \eqref{eqn:main:wdeg:nonlocal} and  \eqref{approx:eqn1} for  %$k^2=\delta^{\sigma}$ 
$k=O(h^{\frac{2\sigma}{4+\sigma}})$ and $\delta= O(h^{\frac{4}{4+\sigma}})$. Then there is $C>0$ independent of $h$ such that  
\begin{align}\label{error_estimate_wdeg}
|u- u_{h}| \leq \left\{
       \begin{array}{ll}
      C \, h^{\frac{4-\sigma}{4+\sigma}} \quad &\mbox{for} \quad 0< \sigma \leq 1, \\[0.2cm]
      C \, h^{\frac{\sigma(4-\sigma)}{4+ \sigma}} \quad &\mbox{for} \quad 1< \sigma <2 .
       \end{array}
  \right.
\end{align}
\end{thm}

\begin{rem}
For $\sigma\leq1$, the results are the same as in Theorem \ref{thm:estimate_sym_main}.
 For $\sigma>1$, the rate of convergence is always more than $\mathcal{O}(h^{\frac{1}{2}})$, and  the rate approaches  $\mathcal{O}(h^{\frac{2}{3}})$ when $\sigma \rightarrow 2$. 
%This is an improvement compare to  Theorem \ref{thm:estimate_sym_main}. 
The "CFL" conditions are the same as in Theorem \ref{thm:estimate_sym_main}.
\end{rem}

\section{Powers of discrete Laplacian} \label{sec:fraclap}

In this section we consider versions of equation \eqref{eqn:main} where the nonlocal operator is the fractional Laplacian, 
\begin{align} \label{eqn:main3_wdeg}
\lambda u(x) + \sup_{\A\in \mathcal{A}} \left\{ f^{\A}(x)  + a^{\A}\,(- \Delta)^{\frac{\sigma}{2}} u(x) \right\} =0.
\end{align}
In other words, $\I^{\A}=-a^{\A}\,(- \Delta)^{\frac{\sigma}{2}}$, $\nu_\alpha(dz)=a^{\A}\frac{c_{N,\sigma}}{|z|^{N+2\sigma}}dz$, and $\eta^\alpha(z)=z$ in \eqref{n_local:term}. Here \ref{A3}--\ref{A7} trivially holds. We assume \ref{C1}, the equation is weakly non-degenerate  (otherwise the equation is purely algebraic), which here is equivalent to
\begin{align}\label{wd}
\text{there is $\alpha_0\in\mathcal A$ such that} \quad a^{\A_0}>0.
\end{align}
Under assumptions \ref{A1}, \ref{A2}, \ref{C1}, and \ref{C3},  we can use Proposition \ref{thm:viscosity_exist}, Lemma \ref{lemma:auxeqn:wdeg:nonlocal}, and Theorem \ref{thm:regularity:vis_soln:wdeg:nonlocal} to conclude wellposedness, stability, approximation, and regularity results for \eqref{eqn:main3_wdeg}.
Here we introduce and analyse a discretization 
\begin{align}\label{eqn:fraclap:aprrox}
  \lambda u_h(x) + \sup_{\A \in \mathcal{A}}\{ f^{\A}(x)  + a^{\A} (-\Delta_h)^{\frac{\sigma}{2}} [u_h](x)\} =0,
\end{align}
based on powers of the discrete Laplacian $(-\Delta_h)^{\frac{\sigma}{2}}$, see \cite{Ciaurri-Stinga-02,  EJT18b} and also \cite{BP15}. As far as we know, this is the first time this type of discretization has been considered for HJB equations. It is a very good approximation in the sense that it is a monotone method of second order accuracy. This is better than the diffusion corrected discretization of Section \ref{monotone scheme}.

Let $\Delta_h\phi(x)=\sum_{k=1}^N\frac1{h^2}\big(\phi(x+he_k)-2\phi(x)-\phi(x-he_k)\big)$ be the 2nd order central finite difference approximation of the Laplacian $\Delta\phi$, then
%Then the power of the discrete Laplacian $(- \Delta_h)^{\frac{\sigma}{2}}$  is defined as 
\begin{align}\label{discrete_fraclap}
(- \Delta_h)^{\frac{\sigma}{2}} \phi(x) := \frac{1}{\Gamma(-\frac{\sigma}{2})} \, \int_0^{\infty} \Big( e^{t\Delta_h} \phi(x) - \phi(x) \Big) \, \frac{dt}{t^{1+ \frac{\sigma}{2}}},
\end{align} 
where $U(t)=e^{t\Delta_h} \psi $ is the solution of semi-discrete heat equation 
\begin{align*}
\partial_t U(x,t) & = \Delta_h \,  U(x, t) \quad \mbox{for} \quad (x,t) \in \rn \times (0,\infty), \\
 U(x,0) & = \psi(x) \quad \mbox{for} \quad x \in \rn. 
\end{align*}  
An explicit formula for $e^{t\Delta_h} \phi$ and details related to this approximation can be found in Section 4.5 of \cite{EJT18b}. We can write approximation \eqref{discrete_fraclap} as a quadrature,
\begin{align*}%\label{approx_fraclap_monotone}
-(- \Delta_h)^{\frac{\sigma}{2}} \phi(x) = \sum_{\bj \in \zn \setminus \{0\}} \Big( \phi(x+ x_{\bj}) - \phi(x)\Big) \kappa_{h,\bj} \quad \mbox{with} \quad  \kappa_{h,\bj} \geq 0.
\end{align*} 
This is obviously a monotone approximation of the fractional Laplacian, and by Lemma 4.22 in \cite{EJT18b},  it has the following local truncation error: 
\begin{lem}\label{lem:trun_err_fraclap}
Assume $\sigma \in (0 ,2)$. Then for any smooth bounded function $\phi$, 
%the consistency bound of $(- \Delta_h)^{\frac{\sigma}{2}}$ is given by %Assume $\phi$ be the smooth and bounded function and consider the monotone discretization of fractional Laplace as \eqref{approx_fraclap_monotone}. Then truncation error of this approximation is 
\begin{align}\label{fraclap_trunc}
\Big| (- \Delta_h)^{\frac{\sigma}{2}} \phi(x) - (- \Delta)^{\frac{\sigma}{2}} \phi(x)\Big| \leq C h^2 \Big(\| D^4 \phi\|_0 + \|\phi\|_0 \Big).
\end{align}
%for any $\sigma \in (0 ,2)$. 
\end{lem} 

%  We now write the approximate equation as 
% \begin{align}\label{eqn:fraclap:aprrox}
%   \lambda u_h(x) + \sup_{\A \in \mathcal{A}}\{ f^{\A}(x)  + a^{\A} (-\Delta_h)^{\frac{\sigma}{2}} [u_h](x)\} =0.
% \end{align}
%For weakly non-degenerate case we consider $c^{\A}(x)= \lambda$ in equation \eqref{eqn:fraclap:aprrox}. \\

We note that Theorem \ref{thm:epprx_exist1} (wellposedness and stability) also holds for \eqref{eqn:fraclap:aprrox}. 
We now state an error bound for this scheme.  The proof is given in Section \ref{sec:prf:flap}. 

\begin{thm}\label{thm:err_bd:flap:wdeg}
Assume $h\in(0,1)$, \ref{A1}, \ref{A2}, \ref{C1}, \ref{C3}, $u$ and $u_h$ are solutions of equation \eqref{eqn:main3_wdeg} and  
the scheme 
\eqref{eqn:fraclap:aprrox}. Then there is $C>0$ such that 
\begin{align}
\|u- u_{h}\|_0 \leq \left\{
       \begin{array}{ll}
      C h^{\frac{1}{2}} \quad &\mbox{for} \quad 0< \sigma \leq 1, \\[0.2cm]
      C h^{\frac{\sigma}{2}} \quad &\mbox{for} \quad 1< \sigma <2 .
       \end{array}
  \right.
\end{align} 
\end{thm} 

\begin{rem}
The problem is weakly non-degenerate and the regularity of the solution can be seen in the rate for $\sigma>1$, cf. Theorem \ref{thm:regularity:vis_soln:wdeg:nonlocal}. This $\sigma$ dependence seems to be optimal, and is consistent as $\sigma \to 2$ with the  $\mathcal{O}(h)$ bound obtained in the 2nd order case in \cite{D-Krylov05}. For $\sigma\in(\frac43,2)$, the rate is better than for the diffusion corrected discretization in Theorem \ref{thm:er_bound:wdeg}. 
% The result is optimal for each $\sigma>1$ and  the convergence rate approaches to the rate obtained for 2nd order case (the rate is of $\mathcal{O}(h)$, c.f. \cite{D-Krylov05}) when $\sigma \rightarrow 2$. 
 \end{rem}

\section{Proofs of the error bounds for monotone quadrature schemes} 
\label{sec_prf_gen}
Here, we give proof of the  convergence rates discussed in Section \ref{monotone scheme}. %% As mentioned before, the main tools to prove convergence rate is Krylov's Shaking the coefficient technique (c.f. \cite{Krylov:2000yy,Krylov:2005lj,BJ05}). 

\subsection{Strongly-degenerate equations -- the proof of Theorem
  \ref{thm:estimate_sym_main}}

Let $\big(\rho_{\epsilon}\big)_{\epsilon>0}$ be the standard mollifier on $\rn$ and define $u_{\epsilon,h}= u_h * \rho_{\epsilon}$. By \eqref{approx:eqn1}, 
\begin{align*}
    f^{\A}(x) + c^{\A}(x) u_h(x) -\mathcal{L}^{\A}_{\delta,k,h}\, u_h(x)  - \sum_{\bj \in \Z^N} \big(u_h(x +x_{\bj})-u_h(x)\big)
\kappa_{h,\bj}^{\A,\delta} \leq 0
\end{align*}
for any $\A \in \mathcal{A}$. Let $f^{\A}_\epsilon=f^{\A}*\rho_\epsilon$, convolve by $\rho_\epsilon$,  to get 
\begin{align*}
    f^{\A}_\epsilon(x) + (c^{\A} u_{h,\epsilon})*\rho_\epsilon(x) -\mathcal{L}^{\A}_{\delta,k,h}\, u_{h,\epsilon}(x)  - \sum_{\bj \in \Z^N} \big(u_{h,\epsilon}(x +x_{\bj})-u_{h,\epsilon}(x)\big)
\kappa_{h,\bj}^{\A,\delta} \leq 0.
\end{align*}
Since $\|f^{\A}*\rho_\epsilon -f^{\A}\|_0 \leq K \epsilon$ and $\|(c^\A u_h)* \rho_\epsilon - c^{\A}\, u_{h,\epsilon}\|_0  \leq \sup_{\A}\|Dc^{\A}\|_0 \|u_h\|_0 \, \epsilon \leq C K^2 \,\epsilon,$
we then find that
\begin{align}\label{prf_err1_new}
 & f^{\A}(x)  +c^{\A}(x) u_{\epsilon,h}(x)   - \I^{\A}[u_{\epsilon,h}](x) \notag\\ 
& \qquad \leq   \big\| \I^{\A}[u_{\epsilon,h}]  - \big(\mathcal{L}^{\A}_{\delta,k,h} \, u_{\epsilon,h} + \I^{\A,\delta}_h[u_{\epsilon,h}]\big)\big\|_0 + (CK^2 +K) \epsilon.
\end{align}
By Lemmas~\ref{lem:err_local_dif_sym_odd},~\ref{lem:local_trunc1},~\ref{lem:trunc_err2}, and
 $|D^ku_{\epsilon,h}|_0 \leq \frac{C\|u_h\|_{0,1}}{\epsilon^{k-1}}$, it follows that
\begin{align}\big| 
&\I^{\A}[u_{\epsilon,h}]  - \big(\mathcal{L}^{\A}_{\delta,k,h} \, u_{\epsilon,h} + \I^{\A,\delta}_h[u_{\epsilon,h}]\big)\big|_0  \nonumber\\
& \label{prf_err_M1}
  \leq M_{\epsilon,\delta} := C \Big(\delta^{4-\sigma} \, \frac{1}{\epsilon^3} 
  + k^2  \, \delta^{2(2-\sigma)} \, \frac{1}{\epsilon^3} + \frac{h^2}{k^2} \frac{1}{\epsilon} + \frac{h^2}{\delta^{\sigma}} \, \frac{1}{\epsilon}  \Big).
\end{align}
Therefore $u_{\epsilon,h}- \frac{C}{\lambda}\tilde{M}_{\epsilon, \delta}$, for $\tilde{M}_{\epsilon,\delta}= M_{\epsilon, \delta} + (CK^2+K)\epsilon$, is a classical (and hence also viscosity) subsolution of equation \eqref{eqn:main}.  By comparison for equation \eqref{eqn:main} (Proposition \ref{thm:viscosity_exist}~(a)),
$u_{\epsilon,h}- \frac{C}{\lambda} \, \tilde{M}_{\epsilon, \delta}\leq u.$
Since $\|u_h - u_{\epsilon,h}\|_0\leq  \epsilon \|Du_h\|_0$, we get
\begin{align}\label{prf:upperbd:gen1}
u_h-u \leq K\big(\epsilon + M_{\epsilon,\delta}\big).
\end{align} 

The bound on $u-u_h$ can be proved in similar way. Let $u_{\epsilon} = u* \rho_{\epsilon}$. Arguing as above, using Lemmas~\ref{lem:err_local_dif_sym_odd},~\ref{lem:local_trunc1}, ~\ref{lem:trunc_err2}, and $\|D^ku_{\epsilon}\|_0 \leq \frac{C\|u\|_{0,1}}{\epsilon^{k-1}}$, we have
 \begin{align*}
&f^{\A}(x)    + c^{\A}(x) \, u_{\epsilon}(x)-\mathcal{L}^{\A}_{\delta,k,h} \, u_{\epsilon}(x) - \I_h^{\A, \delta}[u_{\epsilon}](x) \\
&  \leq   \big\| \I^{\A}[u_{\epsilon}] - \big(\mathcal{L}^{\A}_{\delta,k,h} \, u_{\epsilon} + \I^{\A,\delta}_h[u_{\epsilon}]\big)\big\|_0 + (CK^2+K) \epsilon \leq  M_{\epsilon,\delta} + (CK^2+K) \epsilon.
\end{align*}
 This implies $u_{\epsilon}- \frac{C}{\lambda} \, \tilde{M}_{\epsilon,\delta} $ is a subsolution of the numerical scheme \eqref{approx:eqn1}.  Comparison for the scheme \eqref{approx:eqn1} (Theorem \ref{thm:epprx_exist1}(b)) and $\|u-u_{\epsilon}\|_0< \epsilon \|Du\|_0$  lead to
\begin{align}\label{prf_err1_5}
u_h-u \geq -C  (\epsilon +M_{\epsilon,\delta}).   
\end{align} 

By \eqref{prf:upperbd:gen1} and \eqref{prf_err1_5} we get 
$|u-u_h|\leq C (\epsilon + M_{\epsilon,\delta})$, and then we optimize with respect to $k$, $\delta$, and $\epsilon$. The optimal choices $k^2= O\big(\frac{h \epsilon}{\delta^{2-\sigma}})$ and $\epsilon = O\big(\frac{h}{\delta^{\frac{\sigma}{2}}}\big)$ lead to
%in the expression of $M_{\epsilon,\delta}$ (cf. \eqref{prf_err_M1}) we have that 
\begin{align}\label{prf_err_M2}
 |u-u_h| \leq C \Big(\delta^{4+\frac{\sigma}{2}} h^{-3} 
 + \delta^2 h^{-1} + \frac{h}{\delta^{\frac{\sigma}{2}}} \Big),
\end{align}
and the result follows by choosing $\delta= O \big(h^{\frac{4}{4+\sigma}}\big)$. \hspace{1.1cm}$\square$

\subsection{Intermezzo on regularisations} \label{subsec:frac-heat-ker}

In the remaining proofs we need high order estimates for two different
regularisation procedures: (i) Convolution with standard mollifiers and
(ii) convolution with fractional heat kernels. These estimates are
proved in this section.

Let $\rho_{\varepsilon}(x)= 
\frac{1}{\varepsilon^{N}}\rho\big(\frac{x}{\varepsilon}\big)$ for some
$\rho\in C_c^\infty(\rn)$ with support in $B(0,1)$ and
$\int_{\rn}\rho\, dx =1$. Hence  $\mathrm{supp}\,\rho_{\epsilon}=\overline{B(0,\epsilon)}$  and $\int_{\rn}\rho_{\epsilon}\, dx
=1$. We define
\begin{align}\label{veps1}
  v^{(\epsilon)}= v * \rho_{\epsilon}
\end{align}  
for bounded continuous functions $v$. It then easily follows that
$v^{(\epsilon)}\in C_b^\infty$.

\begin{lem} \label{lem:dif:conv}
If $v \in C^{1, \beta}(\rn)$ for $\beta \in (0,1]$ and $\rho$ is a radial function, then 
$$\|v^{(\epsilon)} -v\|_0 \leq C \epsilon^{1+ \beta}\|v\|_{1,\beta}\qquad\text{and}\qquad \|D^mv^{(\epsilon)} \|_0 \leq
  \frac{K}{\epsilon^{m-1-\beta}}\|v\|_{1,\beta}$$
for any $m\geq 2$, where $C$ and $K$ are independent of $\epsilon$. 
\end{lem}

\begin{proof}
The first inequality follows since $\int_{\rn}y\rho_\epsilon(y)\,dy=0$ and then
\begin{align*}
|v^{(\epsilon)}(x) -v(x)| &= \Big|\int_{\rn} (v(x-y) - v(x) - y\cdot \grad
v(x)) \rho_{\epsilon}(y) \, dy\Big|\\
&\leq C\|v\|_{1,\beta}\int_{\rn}|y|^{1+\beta} \rho_{\epsilon}(y) \,
dy\leq C\|v\|_{1,\beta}\epsilon^{1+\beta}. 
\end{align*}
Since $\int_{\rn}D^{m-1}\rho_{\epsilon}(y) dy =0$ by the divergence theorem,
the second inequality follows since $Dv\in C^\beta$ and
$$D^m v^{(\epsilon)}= Dv * D^{m-1}\rho_{\epsilon} = \int_{\rn}[Dv(x-y)
  - Dv(x)] D^{m-1}\rho_{\epsilon}(y) dy.$$
\end{proof}

Let $\tilde{K}^{\sigma}(t,x):= \mathcal{F}^{-1}\big(
e^{-t|\cdot|^{\sigma}}\big)(x)$ be the fractional heat kernel,
the fundamental solution of the fractional heat equation  $u_t + (-
\Delta)^{\frac{\sigma}{2}} u =0$. Convolution with
$\tilde{K}^{\sigma}$ defines a smooth
approximation of a bounded continuous function $v$, 
\begin{align}\label{veps2}
v^{[\epsilon]}(x) := v(\cdot) * \tilde{K}^{\sigma}(\epsilon^{\sigma},\cdot)(x). 
\end{align}

Let $K^{\sigma}(x)=
\tilde{K}^{\sigma}(1,x)$. To prove estimates on $v^{[\epsilon]}$, we need some well-known properties
of $\tilde{K}^{\sigma}$:
\begin{itemize}\smallskip
\item[(i)]  $\tilde{K}^{\sigma} \in C^{\infty}((0,\infty)\times
  \rn)$, $\tilde{K}^{\sigma}\geq0$, and $\int_{\rn}
  \tilde{K}^{\sigma}(t,x) \, dx =1$ for $t>0$.\smallskip
\item[(ii)] $\tilde{K}^{\sigma}(t+s,x)= \tilde{K}^{\sigma}(t,x) *
\tilde{K}^{\sigma}(s,x)$ for $t,s\geq0$ (convolution semigroup).\smallskip
\item[(iii)] For $t>0$ and $x\in\rn$, $\tilde{K}^{\sigma}(t,x) = t^{-\frac{N}{\sigma}} K^{\sigma}\Big(\dfrac{x}{t^{\frac{1}{\sigma}}}\Big) $ where
$$\frac{c_1 t}{\big( t^{\frac{2}{\sigma}}+  |x|^{2}\big)^{\frac{N+\sigma}{2}}} \leq \tilde{K}^{\sigma}(t,x) \leq \frac{C_2 t}{\big( t^{\frac{2}{\sigma}}+  |x|^{2}\big)^{\frac{N+\sigma}{2}}}.$$ 
\item[(iv)] (Theorem 1.1(c) in \cite{GRZYWNY17}) For any $m>0$ and multi-index $\beta$ with $|\beta|=m$, 
$$|D^{\beta} \, {K}^{\sigma}(x)| \leq  \frac{B_m}{1+
  |x|^{N+\sigma}}\quad\text{for\quad $x\in\rn$}. $$ 
\end{itemize}
We refer to \cite{bogdan10, Droniou03, GRZYWNY17} for the
proofs. 

\begin{lem}\label{lem:appndx_1}
Assume $\epsilon>0$, $\sigma>1$, $\beta \in (\sigma-1,1)$, and $v\in C^{1,
  \beta}(\rn)$. Then there is $C>0$ independent of $\epsilon$,
such that
$$\|v^{[\epsilon]} -v\|_0 \leq C \epsilon^{\sigma}.$$
\end{lem}

\begin{proof}
Let $S_t$ be the fractional heat semigroup,
i.e. $v^{[\epsilon]} = S_{\epsilon^\sigma}(v)$.
Since $\int_{\rn} \tilde{K}^{\sigma}(r,y) dy = 1 $, by Fubini's
Theorem and property (i) above,
\begin{align*}
|(-\Delta)^{\frac{\sigma}{2}}[S_r(v)](x)| & =\Big| \int_{\rn}\int_{\rn} \frac{v(x+z-y) -v(x-y)}{|z|^{N+\sigma}} \tilde{K}^{\sigma}(r,y) dx dy\Big|  \\
& \leq  \int_{\rn}  |(-\Delta)^{\frac{\sigma}{2}}[v](x-y)| \tilde{K}^{\sigma}(r,y) dy \leq \|(-\Delta)^{\frac{\sigma}{2}}[v]\|_0.
\end{align*}
Since $\|(-\Delta)^{\frac{\sigma}{2}}[v]\|_0\leq K\|v\|_{1,\beta}$, for any $t\geq s>0$,
$$|S_t(v) - S_s(v)|=\Big| \int_s^t \partial_r[S_r(v)]dr\Big|=\Big|\int_s^t (-\Delta)^{\frac{\sigma}{2}}[S_r(v)] dr \Big| \leq K(t-s) \|v\|_{1,\beta}.$$ 
The lemma then follows by taking $t= \epsilon^{\sigma}$ and using that
$S_s(v)\to v$ pointwise as $s\to 0$.
\end{proof}
\begin{lem}\label{lem:appndx_2}
Assume $\epsilon>0$, $\sigma>1$, $m\geq 2$, $v\in C^{0, 1}(\rn)$, and define
$\epsilon_1=\frac{\epsilon}{2^{\frac{1}{\sigma}}}$. Then there exists $C>0$ independent of $\epsilon$ such that  
$$\|D^m v^{[\epsilon]}\|_0 \leq \frac{C}{\epsilon^{m-1}}\|v\|_{0,1}\quad \mbox{and} \quad \|D^m v^{[\epsilon]}\|_0 \leq \frac{C}{\epsilon^{m-\sigma}}\|v^{[\epsilon_1]}\|_{1,\sigma-1}.$$
\end{lem}
\begin{proof}
The first estimate is classical and follows from differentiating
$\tilde{K}^{\sigma}$ $(m-1)$ times and $v$ once (c.f. Lemma
\ref{lem:dif:conv}) and noting that
$|x|\big|D^m_x\tilde{K}^{\sigma}(t,x)\big| \in L^{1}(\rn)$ for
$\sigma>1$ and $t>0$ by property (iv) above.

For the second
estimate, we must estimate $\partial_{x_i}D^\alpha v^{[\epsilon]}$ for any
multiindex $\alpha$ with $|\alpha|=m-1$. Rewriting $v^{[\epsilon]}$ as
$$v^{[\epsilon]} = v* \tilde{K}^{\sigma}(\epsilon^{\sigma},\cdot) =v * \tilde{K}^{\sigma}\big(\frac{\epsilon^{\sigma}}{2},\cdot\big) * \tilde{K}^{\sigma}\big(\frac{\epsilon^{\sigma}}{2},\cdot\big)= v^{[\epsilon_1]}*\tilde{K}^{\sigma}\big(\frac{\epsilon^{\sigma}}{2},x\big),$$ 
we find that
$$\partial_{x_i}D^\alpha v^{[\epsilon]} = \partial_{x_i} v^{[\epsilon_1]}*\, D^\alpha \tilde{K}^{\sigma}\big(\frac{\epsilon^{\sigma}}{2},\cdot\big). $$
First, by the divergence theorem and the decay at infinity (property (iv) above), $\int_{\rn}
 D^\alpha\tilde{K}^{\sigma}\big(\frac{\epsilon^{\sigma}}{2},y\big) dy
 =0$. Then, by self-similarity (property (iii)) and $y = \frac\epsilon{2^{\frac1\sigma}} z$,
 $$(D^\alpha_y
 \tilde K)\Big(\frac{\epsilon^{\sigma}}{2},y\Big)= 2^{\frac N\sigma}\frac1{\epsilon^{N+(m-1)}}(D^\alpha_z K)(z).$$
Combining these facts with the change of variables $y = \epsilon z$, we see that
\begin{align*}
|\partial_{x_i}D^\alpha v^{[\epsilon]}|& = \Big|\int_{\rn} \big( \partial_{x_i} v^{[\epsilon_1]}(x-y) - \partial_{x_i} v^{[\epsilon_1]}(x) \big) D^\alpha_y \tilde{K}^{\sigma}\Big(\frac{\epsilon^{\sigma}}{2},y\Big) \, dy \Big| \\
& \leq \frac{2^{\frac N\sigma}}{\epsilon^{m-1}} \int_{\rn} \big|\partial_{x_i} v^{[\epsilon_1]}(x-\epsilon z) - \partial_{x_i} v^{[\epsilon_1]}(x)\big| \left|D^\alpha_z K^{\sigma}(z)\right| \, dz \\
& \leq \frac{K}{\epsilon^{m-\sigma}} \|v^{[\epsilon_1]}\|_{1,\sigma-1} \int_{\rn} |z|^{\sigma-1} \left|D^\alpha_z K^{\sigma}(z)\right| \, dz.
\end{align*}
The proof is complete since 
$|x|^{\sigma-1} \left|D^\alpha K^{\sigma}(x)\right| \in
L^{1}$ by property (iv) above.
\end{proof}

\subsection{Weakly-degenerate equations -- the proof of Theorem
  \ref{thm:er_bound:wdeg}}

We first prove a discrete version of the bound on
nonlocal operator in Theorem \ref{thm:reg:wdeg:nonlocal}. Then we
show that these bounds leads to regularity of the numerical solution. From regularity,
approximation, and comparison arguments the error bounds
follows. Regularization arguments and the results of the previous section 
are used throughout. For $h,k,\epsilon>0$, and $\delta\in(0,1)$,  we define  
\begin{align*}
&\hat\I^{\A}_{\delta,k,h}[\phi] := \mathcal{L}^{\A}_{\delta,k,h}[\phi] +
\I^{\A,\delta}_h [\phi],\\
&\J^{\A,\delta}_h[\phi] := \sum_{\bj \in \zn}\big(\phi(x+ x_{\bj}) -\phi(x)\big) \int_{|z|>\delta} \omega_{\bj}(\eta^{\A}(z)) \frac{dz}{|z|^{N+\sigma}},
\end{align*}
where $\mathcal{L}^{\A}_{\delta,k,h}$,
$\I^{\A,\delta}_h$, and the weight function $\omega_{\bj}$ are defined in section \ref{subsec:discr}. \ee
By
definition $\J^{\A,\delta}_h$ is a monotone approximation of the
non-singular part of the operator \begin{align}\label{calligraphy-J}
    \J^{\A}[\phi]:= \int_{\rn}
\big(\phi(x+\eta^{\A}(z)) -\phi(x) - \grad \phi(x)\cdot
\eta^{\A}(z)1_{|z|<\delta}\ee\big) \frac{dz}{|z|^{N+\sigma}}
\end{align}
with local truncation error (Taylor expand, see e.g. \cite[Section 3]{BCJ1})
\begin{align}\label{truncbd_nonlocal_wosing}
|\J^{\A,\delta}_h[\phi](x) - \J^{\A}[\phi](x)| \leq C (\be\|\phi\|_0 + \|D^2 \phi\|_0) \big( \delta^{2-\sigma}+ h^{2}\delta^{-\sigma} \big).
\end{align}  
%for any $\A\in \mathcal{A}$ and $x\in \rn$.
The discrete
version of Theorem \ref{thm:reg:wdeg:nonlocal} is the following result.
%%% ERJ
\begin{thm}\label{thm:reg:approx:wdeg:nonlocal}
Assume \ref{A1}-\ref{A5}, \ref{C1}-\ref{C3}, and $u_{h}$ solves \eqref{eqn:approx:wdeg:nonlocal}. Then for $\delta\in(0,1)$, $\delta\geq h$, and $k\geq \delta^{\frac{\sigma}{2}}$, there is a $K>0$ independent of $h,k,\delta$ such that
%\begin{itemize}
%\item[(i)] 
\begin{align}\label{regbnd:approx:aux:nonlocal}
&\|\, \hat\I^{\A_0}_{\delta,k,h} [u_{h}]\, \|_0 \leq K, \\
\label{regbnd:approx:aux:frac_nonlocal}
&\|\,  \J^{\A_0,\delta}_h [u_{h}]\, \|_0 \leq \frac{K}{c_{ \A_0}}.
\end{align}
%\end{itemize}
\end{thm}

The proof relies on the following technical lemma.

\begin{lem}\label{lem:regularize_f}
Assume \ref{A1}-\ref{A6}, \ref{C1}-\ref{C3}, and $\A_0$ is defined in \ref{C1}. For $\sigma\in(0,1)$, there is a $K>0$ independent of $\delta,h,k$ such that 
\begin{align*}
\|\mathcal{L}^{\A_0}_{\delta,k,h}[f^{\A}]\|_0 \leq K \Big[ \frac{h^\sigma}{k^2}+ k^{\sigma-2} \delta^{\frac{\sigma(2-\sigma)}{2}}\Big] \|f^{\A}\|_{1, \sigma-1}.
\end{align*}
\end{lem}

\begin{proof}
  Let $f^{\A}_{(\gamma)}:=
f^{\A}*\rho_{\gamma} \in C^{\infty}_b(\rn)$. By Lemma
\ref{lem:dif:conv} and the fact that $f^{\A}\in
C^{1,\sigma-1}(\rn)$ by \ref{C3},
\begin{align}\label{regularize_f_prop}
 \|D^{m}f^{\A}_{(\gamma)}\|_0 \leq \frac{C \|f^{\A}\|_{1, \sigma-1}}{\gamma^{m-\sigma}}\quad  \mbox{and}\quad  \|f^{\A}-f^{\A}_{(\gamma)}\|_0\leq C \gamma^{\sigma} \|f^{\A}\|_{1, \sigma-1}.
\end{align}
Then by \eqref{SL_bd_interpolant}, \eqref{SL_approx_term}, the bound
on $a^{\A}_\delta$ in \eqref{a-bnd}, and first part of \eqref{regularize_f_prop}, 
\begin{align*}
 \Big|\mathcal{L}^{\A_0}_{\delta,k,h}[f^{\A}_{(\gamma)}]\Big| &\leq \Big|\mathcal{D}^{\A_0}_{\delta,k}[f^{\A}_{(\gamma)}]\Big|  +  \frac{C h^2}{k^2} \|D^2f^{\A}_{(\gamma)}\|_0 \notag\\
& \leq  K|(\sqrt{a^{\A}_{\delta}})_{i})|^2 \|D^2f^{\A}_{(\gamma)}\|_0
 + C \frac{h^2}{k^2} \|D^2f^{\A}_{(\gamma)}\|_0 \notag \\
 &\leq  \frac{K}{\gamma^{2-\sigma}} \Big( \delta^{2-\sigma} +
 \frac{h^2}{k^2}\Big) \|f^{\A}\|_{1,\sigma-1}. %\label{regularize_bd_f_1}
\end{align*}
By the second part of \eqref{regularize_f_prop} and the definition of
$\mathcal{L}^{\A_0}_{\delta,k,h}$ in \eqref{local_epprox_term},
\begin{align*}%\label{regularize_bd_f_2}
\big|\mathcal{L}^{\A_0}_{\delta,k,h}[f^{\A}_{(\gamma)}] - \mathcal{L}^{\A_0}_{\delta,k,h}[f^{\A}]\big| 
= \big|\mathcal{L}^{\A_0}_{\delta,k,h}[f^{\A}_{(\gamma)} -f^{\A}] \big| \leq K \frac{\gamma^{\sigma}}{k^2}\|f^{\A}\|_{1,\sigma-1},
\end{align*} 
and then % by  \eqref{regularize_bd_f_1} and \eqref{regularize_bd_f_2}
\begin{align}\label{upperbd_regularized_local_wdeg}
\|\mathcal{L}^{\A_0}_{\delta,k,h}[f^{\A}]\|_0  
\leq & \|\mathcal{L}^{\A_0}_{\delta,k,h}[f^{\A}_{(\gamma)}]\|_0 + \|\mathcal{L}^{\A_0}_{\delta,k,h}[f^{\A}_{(\gamma)}] - \mathcal{L}^{\A_0}_{\delta,k,h}[f^{\A}]\|_0 \notag \\
\leq & \,  K\Big[\frac{1}{\gamma^{2-\sigma}}\Big(\frac{h^2}{k^2}+ \delta^{2-\sigma}\Big) + \frac{\gamma^{\sigma}}{k^2} \Big]\|f^{\A}\|_{1, \sigma-1}.
\end{align}
The result follows by taking $\gamma = \max\{h, k\,\delta^{\frac{2-\sigma}{2}}\}$.
\end{proof}

\begin{proof}[Proof of Theorem \ref{thm:reg:approx:wdeg:nonlocal}]
(i)\quad  Since $u_h$ solves 
  \eqref{eqn:approx:wdeg:nonlocal}, we find as in the proof of Theorem
  \ref{thm:reg:wdeg:nonlocal}, that $-\I^{\A_0,\delta}_h
        [u_{h}]$ is a supersolution of  
\begin{align}\label{approx:wdeg:reg:eqn:nonlocal}
\lambda \, v(x) + \sup_{\A \in \mathcal{A}} \left\{- \mathcal{L}^{\A}_{\delta,k,h}[v]-  \, \I^{\A,\delta}_h [v]  -\I^{\A_0,\delta}_h[f^{\A}](x)  \right\} =0.
\end{align}
By assumptions \ref{C3} and \ref{A3},
$$\|\I^{\A_0,\delta}_h [f^{\A}]\|_0
\leq C_1 := \|f^\alpha\|_{1,\beta-1}\int_{|z|<1}|z|^\beta
\nu_\alpha(dz)+2\|f^{\A}\|_0 \int_{|z|\geq1}\nu_\alpha(dz),$$
where the constant $C_1\geq0$ is independent of $\alpha$, $\delta$,
and $h$. Since 
$-\frac{C_1}{\lambda}$ is a subsolution of
\eqref{approx:wdeg:reg:eqn:nonlocal}, the comparison principle yields that
%\begin{align} \label{regbd:approxfrac:nonlocal}
$\I^{\A_0,\delta}_h [u_h](x) \leq \frac{C_1}{\lambda}.$ 
%\end{align}
Arguing in the same way for the operator
$\mathcal{L}^{\A_0}_{\delta,k,h}$ and using Lemma \ref{lem:regularize_f}, we  get that 
\begin{align*}
\mathcal{L}^{\A_0}_{\delta,k,h}[u_h]  \leq \frac{K}{\lambda} \Big[ \frac{h^\sigma}{k^2}+ k^{\sigma-2} \delta^{\frac{\sigma(2-\sigma)}{2}}\Big] \|f^{\A}\|_{1, \sigma-1}.
\end{align*}
Taking $k \geq C \max\{\delta^{\frac{\sigma}{2}},
h^{\frac{\sigma}{2}}\}= C \delta^{\frac{\sigma}{2}}$ (assuming $\delta
\geq h$) we find a constant $C_2\geq0$ independent of $\alpha$, $k$, $h$, and $\delta$ such that
%\begin{align}\label{upperbd_local_wdeg}
$\mathcal{L}^{\A_0}_{\delta,k,h}[u_h] \leq \frac {C_2}\lambda.$ 
%\end{align}
Combining the two estimates then gives
$$\hat\I^{\A_0}_{\delta,k,h}[u_h](x)\leq \frac{C_1+C_2}\lambda.$$

To get the lower bound, we use the definition of  $\hat\I^{\A_0}_{\delta,k,h}[u_h]$ and
the fact that $u_h$ is a subsolution of
\eqref{eqn:approx:wdeg:nonlocal}, to see that
\begin{align*}
&-\,\hat\I^{\A_0}_{\delta,k,h} [u_h](x)   \leq \sup_{\A\in \mathcal{A}}\{- \hat\I^{\A}_{\delta,k,h}[u_h](x)\} \\
& \leq \lambda \, u_h(x)+  \sup_{\A \in \mathcal{A}} \left\{- \mathcal{L}^{\A}_{\delta,k,h}[u_h]- \,  \I^{\A,\delta}_h [u_h](x) + f^{\A}(x) \right\} + \Big( \lambda\|u_h\|_0 + \|f^{\A}\|_0\Big) \\
& \leq  \Big(\lambda \|u_h\|_0 + \|f^{\A}\|_0\Big).
\end{align*}
In view of \ref{A2} and Theorem \ref{thm:epprx_exist1} this completes the proof of \eqref{regbnd:approx:aux:nonlocal}.
\medskip

\noindent (ii) \quad The upper bound for
$-\J^{\A_0,\delta}_h[u_h]$ follows from the same reasoning that led to
the upper bound in part~(i). To prove the lower bound,  we first note
that $\int_{\delta<|z|<1}\omega_{\bj}(\eta^{\A_0}(z))
\Big(\frac{d\nu_{\A_0}}{dz}(z) - \frac{c_{\A_0}}{|z|^{N+\sigma}}\Big)\, dz
\geq 0$ by \ref{C1}$(i)$ and the fact $\omega_{\bj}\geq 0$. By 
arguments similar to those that led to estimate \eqref{prf_err_M3}, we
then find that
\begin{align*}
\sum_{\bj \in \mathbb Z^N} \big( u_h(x+x_{\bj}) - u_h(x) \big)  \int_{\delta<|z|<1}\omega_{\bj}(\eta^{\A_0}(z)) \Big(\frac{d\nu_{\A_0}}{dz}(z) - \frac{c_{\A_0}}{|z|^{N+\sigma}}\Big)\, dz \leq \frac{K}{\lambda} . 
\end{align*}
Then by \eqref{regbnd:approx:aux:nonlocal} (this bound also holds for
$\I^{\A_0,\delta}_h[u_h]$, see the proof), $\sum_{\bj \in 
  \zn}\omega_{\bj}(\eta^{\A_0}(z)) =1$, and \ref{A4}, we have 
\begin{align*}
-\J^{\A_0,\delta}_h[u_h]& \leq \frac{K}{\lambda} - \I^{\A_0,\delta}_h[u_h]\\
&\quad+\sum_{\bj \in \mathbb Z^N} \big( u_h(x+x_{\bj}) - u_h(x) \big)  \int_{|z|>1}\omega_{\bj}(\eta^{\A_0}(z)) \Big(\frac{d\nu_{\A_0}}{dz}(z) - \frac{c_{\A_0}}{|z|^{N+\sigma}}\Big)\, dz \\
& \leq 
K + C \|u_h\|_0\Big(\int_{|z|>1} \nu_{\A_{0}}(dz) + \int_{|z|>1}\frac{c_{\A_0} dz}{|z|^{N+\sigma}}\Big) \leq \, K + C\|u_h\|_0. 
\end{align*} 
 This completes the proof.
\end{proof}

By Theorem \ref{thm:regularity:vis_soln:wdeg:nonlocal} the solution $u$ of \eqref{eqn:main:wdeg:nonlocal} and its regularization $u^{(\epsilon)}$ satisfy the bounds of Lemma \ref{lem:dif:conv} with $\beta = \sigma -1$.  We now show similar bounds for the solution $u_h$ of 
%and $u_h^{(\epsilon)}$ corresponding to 
the scheme \eqref{eqn:approx:wdeg:nonlocal} and regularizations of $u_h$. The results will incorporate error terms due to truncation bounds for approximate operators. 
\begin{lem} \label{lem:fracbd:conv:nonlocal}
	Assume \ref{A1}-\ref{A7}, \ref{C1}-\ref{C3}, $\delta\in(0,1)$, $\delta\geq h$, $u_h$ solves \eqref{eqn:approx:wdeg:nonlocal}, and $\tilde u_h= u_h * \phi$ for $0\leq \phi\in C^\infty(\rn)$ with $\int_{\rn}\phi \,dx=1$. Then there are $K_1,K_2>0$ independent of $\delta,k,h$ and $\phi$ such that
	\begin{align*}
%	 &(i)\quad \\[0.2cm]
  &(i)\quad\|(-\Delta)^{\frac{\sigma}{2}}[\tilde u_h]\|_0  \leq K_1 \Big( \|u_h\|_{0,1}+ \delta^{2-\sigma} (\|u_h\|_0  + \|D^2\tilde u_h\|_0 )\Big), %\label{frac:bd:conv}	
  \notag \\[0.2cm]
 &(ii)\quad\|\tilde{u}_h\|_{1,\sigma-1} \leq K_2  \Big( 1+ \|u_h\|_{0,1} + \delta^{4-\sigma} \|D^4 \tilde{u}_h\|_0 + \delta^{2(2-\sigma)} k^2 \|D^4\tilde{u}_h\|_0  \notag \\
& \hspace*{6cm}+  \frac{h^2}{k^2} \|D^2 \tilde{u}_h\|_0 + h^{2} \delta^{-\sigma}\|D^2 \tilde{u}_h\|_0 \Big).
	\end{align*} 	
\end{lem}
Note that a bound like (ii) follows from (i) by elliptic regularity, but bound (ii) is an improvement on any bound coming from (i).
\smallskip
\begin{proof}
(i) \  
 Note that $\eta^{\alpha_0}(0)=0$ by \ref{A3}, $z-\eta^\A_0(z)=\mathcal{O}(|z|^2)$ by \ref{C1}$(ii)$, 
%$\int_{|z|>1}   \frac{\eta^{\A_0}(z)\,dz}{|z|^{N+\sigma}}=0$ (by \ref{A5}), 
and $\|\tilde u_h\|_{0,1}\leq \|u_h\|_{0,1}$ by properties of convolutions. By the definition of $\J^{\A_0}$  \eqref{calligraphy-J}, assumptions \ref{A3}, \ref{A5}--\ref{A7}, \ref{C1}, and the truncation error bound \eqref{truncbd_nonlocal_wosing}, 
\begin{align*}
&|(-\Delta)^{\frac{\sigma}{2}}[\tilde{u}_h](x)| \\
&\leq  |\J^{\A_0}[\tilde{u}_h](x)|+ \Big| \int_{|z|<1}  (z-\eta^{\A_0}(z)) \cdot \grad \tilde{u}_h(x)\, \frac{dz}{|z|^{N+\sigma}}\Big|  \\
& \quad  
+ \Big| \int_{|z|>1} \tilde{u}_h (x+z) -  \tilde{u}_h (x +  \eta^{\A_0}(z)) \frac{dz}{|z|^{N+\sigma}} \Big| \\
& \leq \|\J^{\A_0}[\tilde{u}_h]\|_0 +   \|\grad \tilde{u}_h\|_{0} \int_{|z|<1}\frac{K|z|^2\,dz}{|z|^{N+\sigma}} + 2\|\tilde{u}_h\|_0 \int_{|z|>1}\frac{dz}{|z|^{N+\sigma}} \\
&\leq \|\J_h^{\A_0,\delta}[\tilde{u}_h]\|_0  + C (\|u_h\|_0 +  \|D^2 \tilde u_h\|_0) \big( \delta^{2-\sigma}+ h^{2}\delta^{-\sigma} \big)  +c_\sigma \|u_h\|_{0,1}.
\end{align*}
The proof is complete since by Theorem \ref{thm:reg:approx:wdeg:nonlocal} and properties of convolutions,   
 $$ \|\J^{\A_0,\delta}_h [\tilde{u}_h]\|_0 \leq C \|\J^{\A_0,\delta}_h[u_h]\|_0\leq K. $$

\noindent (ii) \ By Theorem \ref{thm:reg:approx:wdeg:nonlocal} and  properties of convolutions,  
  $\|\hat{\mathcal{I}}^{\A_0}_{\delta,k,h}  [\tilde{u}_h]\|_0 \leq C \|\hat\I^{\A_0}_{\delta,k,h} [u_h]\|_0\leq K.$ 
	From the error bounds \eqref{err_local_dif_sym}, \eqref{local_trunc1}, and \eqref{nonl_trunc1}, it then follows that
	\begin{align}\label{esti1_lem_reg:wdeg}
	&\|\mathcal{I}^{\A_0} [\tilde u_h]\|_0  \leq K_1\Big( K 
 + \delta^{4-\sigma} \|D^4 \tilde u_h\|_0 + \delta^{2(2-\sigma)} k^2 \|D^4\tilde  u_h\|_0  \notag \\ 
	& \hspace*{5.4cm} + \frac{h^2}{k^2} \|D^2 \tilde u_h\|_0+ \frac{h^{2}}{\delta^{\sigma}} \|D^2 \tilde u_h\|_0 \Big).
	\end{align}
We define the operator
\begin{align*}
\widetilde{\mathcal{J}}[\phi](x) := & \int_{|z|<1} \big( \phi(x+z) -\phi (x) -z\cdot \nabla \phi (x) \big) \nu^{\A_0}(dz) \\
 & \hspace{2cm} + \int_{|z|>1} \big( \phi(x+z) -\phi (x) -z\cdot \nabla \phi (x) \big) \frac{c_{\A_0}}{|z|^{N+\sigma}} .    
\end{align*}
Since $z-\eta^\A_0(z)=\mathcal{O}(|z|^2)$ by \ref{C1}$(ii)$ and 
$\|\tilde u_h\|_{0,1}\leq \|u_h\|_{0,1}$, by \eqref{esti1_lem_reg:wdeg} we have
\begin{align} 
    &\big|\widetilde{\mathcal{J}}[\tilde{u}_h](x)\big| \leq  \big| \mathcal{I}^{\A_0}[\tilde{u}_h](x)\big| \notag\\
    &\quad +\Big| \int_{|z|<1} \big( \tilde{u}_h(x+z) -\tilde{u}_h (x+ \eta^{\A_0}(z)) -(z- \eta^{\A_0}(z))\cdot \nabla \tilde{u}_h (x) \big) \nu^{\A_0}(dz)\Big|  \notag \\
    & \quad +  \Big|\int_{|z|>1} \big( \tilde{u}_h(x+z) -\tilde{u}_h(x) \big)  \frac{c_{\A_0} dz}{|z|^{N+\sigma}} -\int_{|z|>1} \big( \tilde{u}_h(x+j^\A(z)) -\tilde{u}_h(x) \big) \nu^{\A_0}(dz) \Big|\notag \\
    &\leq  \big|\mathcal{I}^{\A_0}[\tilde{u}_h](x)\big| + 2\|\grad \tilde{u}_h\|_0 \int_{|z|<1} |z-\eta^{\A_0}(z)|\, \nu^{\A_0}(dz) + 2 \|\tilde{u}_h\|_0 \int_{|z|>1} \frac{(c_{\A_0}+C)}{|z|^{N+\sigma}} \notag \\
    & \leq C\Big( K + \delta^{4-\sigma} \|D^4 \tilde{u}_h\|_0 + \delta^{2(2-\sigma)} k^2 \|D^4\tilde{u}_h\|_0  \notag \\
& \hspace*{5cm}+  \frac{h^2}{k^2} \|D^2 \tilde{u}_h\|_0 + \frac{h^{2}}{\delta^{\sigma}} \|D^2 \tilde{u}_h\|_0  + \|u_h\|_{0,1}\Big). \label{esti2_lem_reg:wdeg}
\end{align}
Hence $\widetilde{\mathcal{J}}[\tilde{u}_h] \in L^{\infty}(\rn)$ for fixed $\delta$ and $h$. 
 By \ref{C1}$(i)$ and \ref{A6}, the assumptions of the regularity result \cite[Theorem 3.8]{DRSV-regularity} are satisfied, and we conclude that
\begin{align*}
\|\tilde{u}_h\|_{1,\sigma-1} \leq K \Big( \|\tilde{u}_h\|_{0} + \|\widetilde{\mathcal{J}} [\tilde{u}_h]\|_0\Big). 
\end{align*}
The result then follows from \eqref{esti2_lem_reg:wdeg}. 	
\end{proof}

We now give results approximation and derivative bounds mollifications of $u_h$ by the fractional heat kernel. These are discrete versions of Lemmas \ref{lem:appndx_1} and \ref{lem:appndx_2}. 

\begin{lem}\label{lem:bd:der:numsol_reg:wdeg:nonlocal}
	 Assume $\delta\in(0,1)$, $h\leq \delta$, $\epsilon>0$, \ref{A1}-\ref{A5}, \ref{C1}-\ref{C3}, $u_h$ solves \eqref{eqn:approx:wdeg:nonlocal}, and its mollification $u_h^{[\varepsilon]}$ is defined in \eqref{veps2}. Then for $m\geq 2$,  \begin{align*}
	 %\label{improved:bd:derivative:approx}
	 \|D^m u_h^{[\varepsilon]}\|_0 \leq K \frac{\|u_h\|_{0,1}}{\varepsilon^{m-\sigma}}\Big(1+  \big(\delta^{4-\sigma}  + \delta^{2(2-\sigma)} k^{2} \big) \frac{1}{\varepsilon^3} + \big(h^2 k^{-2} + h^2 \delta^{-\sigma} \big) \frac{1}{\varepsilon} \Big).
\end{align*}
\end{lem}
\begin{proof}
By Lemma \ref{lem:appndx_2} and Lemma \ref{lem:fracbd:conv:nonlocal} $(ii)$ with $\phi (x) = \tilde{K}^\sigma(\varepsilon^\sigma,x)$,
\begin{align*}
\|D^m u_h^{[\varepsilon]}\|_0 &\leq \frac{K}{\epsilon^{m-\sigma}} \|u_h^{[\varepsilon_1]}\|_{1,\sigma-1} \\ & \leq \frac{K}{\varepsilon^{m-\sigma}}\Big(1+ \|u_h\|_{0,1} + \delta^{4-\sigma} \|D^4 u_h^{[\varepsilon_1]}\|_0 + \delta^{2(2-\sigma)} k^2 \|D^4u_h^{[\varepsilon_1]}\|_0  \notag \\
& \hspace*{4,5cm}+  \frac{h^2}{k^2} \|D^2 u_h^{[\varepsilon_1]} \|_0 + \frac{h^{2}}{\delta^{\sigma}} \|D^2 u_h^{[\varepsilon_1]}\|_0 \Big),
\end{align*} 
where $\varepsilon_1= \frac{\varepsilon}{2^{\frac{1}{\sigma}}}$. The result then follows from the first part of Lemma \ref{lem:appndx_2}. 
%and by choosing $k^2 = \frac{h\varepsilon}{\delta^{2-\sigma}}$. 
\end{proof}

\begin{lem} \label{thm:regularization_bd:wdeg:nonlocal}
Assume $0<h\leq \delta\leq \epsilon$, $\delta\in(0,1)$, \ref{A1}-\ref{A5}, \ref{C1}-\ref{C3}, $u_h$ solves \eqref{eqn:approx:wdeg:nonlocal}, and its mollification $u_h^{[\varepsilon]}$ is defined in \eqref{veps2}. Then 
\begin{align*}
 \|u_h^{[\varepsilon]} - u_h\|_0  \leq 
 C\big(\delta+\varepsilon^\sigma + \delta^{2-\sigma}\varepsilon^{2(\sigma-1)}+ k^2\delta^{1-\sigma}  + \frac{h^2}{k^2}\delta^{\sigma-1}\big). 
\end{align*}
\end{lem}

\begin{proof}
Let $ S_t $ be the fractional heat semigroup (c.f. the proof of Lemma \ref{lem:appndx_1}) so that $u_h^{[\varepsilon]} = S_{\varepsilon^\sigma}(u_h)$. By properties of $S_t$ and Lemmas \ref{lem:fracbd:conv:nonlocal} $(i)$ and \ref{lem:bd:der:numsol_reg:wdeg:nonlocal}, we have 
% $\big[(\frac{r}{2})^{\frac{1}{\sigma}}\big]$
\begin{align*}
& |S_t[u_h] - S_s[u_h]| =  \Big|\int_s^t (-\Delta)^{\frac{\sigma}{2}}\big[S_r[u_h]\big] dr \Big| \leq C\int_s^t\big(1 +\delta^{2-\delta}(1+{\|D^2u_h^{[r^{\frac{1}{\sigma}}]}\|}_0)\big) d r\\
& \leq C\int_s^t \Big(1+ \frac{\delta^{2-\sigma}}{r^{\frac{2-\sigma}{\sigma}}}\Big(1+ \frac{\delta^{4-\sigma} + \delta^{2(2-\sigma)}k^2}{r^{\frac{3}{\sigma}}} + \frac{h^2k^{-2}+ h^2\delta^{-\sigma}}{r^{\frac{1}{\sigma}}} \Big)\Big) d r  \\
%& \leq  C \Big[ t+ \int_s^t \Big(\frac{\delta^{6-2\sigma} + \delta^{6-3\sigma}k^2}{r^{\frac{5-\sigma}{\sigma}}} + \frac{h^2\delta^{2-2\sigma}+ h^2\delta^{2-\sigma}k^{-2}}{r^{\frac{3-\sigma}{\sigma}}} \Big) \, d r + \int_s^t \frac{\delta^{2-\sigma}}{r^{\frac{2-\sigma}{\sigma}}} d r \Big] \\
& \leq C\Big(t +  \delta^{2-\sigma} t^{\frac{2(\sigma-1)}{\sigma}} + \big(\delta^{6-2\sigma} + \delta^{6-3\sigma}k^2\big) s^{\frac{2\sigma-5}{\sigma}} +\big( h^2\delta^{2-2\sigma}+ h^2\delta^{2-\sigma}k^{-2}\big)s^{\frac{2\sigma-3}{\sigma}} \Big). 
\end{align*}
Since  $ \|S_s[u_h]-[u_h]\|_0 \leq Cs^{\frac{1}{\sigma}}{\|u_h\|}_{0,1}$, we then find that
\begin{align*}
\|S_t[u_h]-[u_h]\|_0 & \leq C\Big(s^{\frac{1}{\sigma}} + t +  \delta^{2-\sigma} t^{\frac{2(\sigma-1)}{\sigma}} + \big(\delta^{6-2\sigma} + \delta^{6-3\sigma}k^2\big) s^{\frac{2\sigma-5}{\sigma}} \\
& \hspace{4cm} +\big( h^2\delta^{2-2\sigma} + h^2\delta^{2-\sigma}k^{-2}\big)s^{\frac{2\sigma-3}{\sigma}}\Big).
\end{align*}
This estimate holds for any  $s\in(0,t)$. Note that since $h\leq\delta$, the Take $t = \varepsilon^\sigma$ and $s = \delta^{\sigma}$ to find that
% \begin{align*}
% \|S_t[u_h]-[u_h]\|_0 & \leq C \Big( t + \delta^{2-\sigma} t^{\frac{2(\sigma-1)}{\sigma}} + \big(\delta^{6-2\sigma} + \delta^{6-3\sigma}k^2\big) \delta^{2\sigma-5} \\
% & \hspace{4cm} +\big( h^2\delta^{2-2\sigma} + h^2\delta^{2-\sigma}k^{-2}\big)\delta^{2\sigma-3}\Big).
% \end{align*}
% Now; write $t = \varepsilon^\sigma  $ and we have 
\begin{align*}
\|u_h^{[\varepsilon]} - u_h\|_0 & \leq C\Big(\delta+ \varepsilon^\sigma + \delta^{2-\sigma}\varepsilon^{2(\sigma-1)} + \big(\delta^{6-2\sigma} + \delta^{6-3\sigma}k^2\big) \delta^{2\sigma-5} \\
& \hspace{4,5cm} +\big( h^2\delta^{2-2\sigma} + h^2\delta^{2-\sigma}k^{-2}\big)\delta^{2\sigma-3}\Big).\\
&\leq C(\delta+\varepsilon^\sigma + \delta^{2-\sigma}\varepsilon^{2(\sigma-1)}+ \delta+ k^2\delta^{1-\sigma} + \delta + \frac{h^2}{k^2}\delta^{\sigma-1}).
\end{align*}
This completes the proof. 
%{\be\bf Why this choice of $s$? Is it optimal?? (Seems not...) Same as Krylov??} Because it is good enough.
%Obs that since $h\leq\delta$, $h^2\delta^{2-2\sigma}s^\frac{2\sigma-3}\sigma$ is a smaller term than $\delta^{6-2\sigma} s^{\frac{2\sigma-5}{\sigma}}$ ($s=\delta^\sigma$ is minimizing the sum of the terms and $s^{\frac1\sigma}$). Seems optimal to take 
%$$s=\max(\delta^\sigma,\delta^{\frac{\sigma(6-3\sigma)}{6-2\sigma}}k^{\frac{2\sigma}{6-2\sigma}},\delta^{\frac{\sigma(2-2\sigma)}{4-2\sigma}}(\frac hk)^{\frac{2\sigma^2}{4-2\sigma}})$$
%which leads to 
%\begin{align*}
% \|u_h^{[\varepsilon]} - u_h\|_0 & \leq C\Big(  \varepsilon^\sigma + \delta^{2-\sigma}\varepsilon^{2(\sigma-1)} +\delta + \delta^\frac{6-3\sigma}{6-2\sigma}k^{\frac2{6-2\sigma}} \\
% & \hspace{4,5cm} +\delta +\delta^{\frac{2-2\sigma}{4-2\sigma}}(\frac hk)^{\frac{2\sigma}{4-2\sigma}}.
% \end{align*}
\end{proof}
In the last proof the dependence on the parameters are only partially optimized, but the result is still good enough for our purposes -- the optimal error bound that we will prove next.
\subsubsection*{Proof of Theorem \ref{thm:er_bound:wdeg}.}
The proof is similar to the proof of Theorem \ref{thm:estimate_sym_main}, and only the case  $\sigma>1$ is new.  
 Let $\big(\rho_{\epsilon}\big)_{\epsilon>0}$ be the standard mollifier on $\rn$ and define $u^{(\epsilon)}= u * \rho_{\epsilon}$. Since $u$ is the viscosity solution of \eqref{eqn:main:wdeg:nonlocal}, $u^{(\epsilon)}$ is a smooth solution of 
	\begin{align*}
	\lambda \, u^{(\epsilon)} + \sup_{\A \in \mathcal{A}} \left\{ (f^{\A})^{(\epsilon)}(x)- \I^{\A} [u^{(\epsilon)}]  \right\} \leq 0 .
	\end{align*}
	By Theorem \ref{thm:regularity:vis_soln:wdeg:nonlocal}, $u \in C^{1, \sigma-1}(\rn)$, and by \ref{C3} and Lemma \ref{lem:dif:conv}, $\|f^{\A}-(f^{\A})^{(\varepsilon)}\|_0 \leq K \varepsilon^{\sigma}$. Therefore, from the truncation error bounds \eqref{err_local_dif_sym}, \eqref{local_trunc1} and \eqref{nonl_trunc1} we get
	\begin{align*}
	\lambda \, u^{(\epsilon)} + & \sup_{\A \in \mathcal{A}} \left\{ f^{\A}(x) - \I^{\A}_h u^{(\epsilon)}  \right\}  \leq \sup_{\A\in \mathcal{A}} \Big[\|f^{\A} - (f^{\A})^{(\epsilon)}\|_0 +  \| \I^{\A}_h [u^{(\epsilon)}] - \I^{\A} [u^{(\epsilon)}]\|_0\Big]\\
	& \leq C\epsilon^\sigma + C \Big(\delta^{4-\sigma} \|D^4 u^{(\varepsilon)}\|_0 + \delta^{2(2-\sigma)} k^2 \|D^4u^{(\varepsilon)}\|_0  \notag \\ 
	& \hspace*{3cm} + \frac{h^2}{k^2} \|D^2 u^{(\varepsilon)}\|_0+ \frac{h^2}{\delta^{\sigma}} \|D^2 u^{(\varepsilon)}\|_0\Big) \\
	& \leq C \Big(\epsilon^{\sigma} + \delta^{4-\sigma}\frac{1}{\varepsilon^{4-\sigma}} + \delta^{2(2-\sigma)}k^2\frac{1}{\varepsilon^{4-\sigma}} + \frac{h^2}{k^2} \frac{1}{\varepsilon^{2-\sigma}} + \frac{h^2}{\delta^{\sigma}} \frac{1}{\varepsilon^{2-\sigma}}\Big):= A_{\varepsilon}.
	\end{align*}
	 Hence $u^{(\epsilon)} - \frac{C}{\lambda}A_{\varepsilon}$ is a subsolution of the equation \eqref{eqn:approx:wdeg:nonlocal}, and the comparison principle for \eqref{eqn:approx:wdeg:nonlocal} then implies that
	%\begin{align*}
	$u^{(\epsilon)} - \frac{C}{\lambda} A_{\varepsilon} \leq u_h$. 
	%\end{align*} 
	By Theorem \ref{thm:regularity:vis_soln:wdeg:nonlocal} and Lemma~\ref{lem:dif:conv},   $\|u^{(\epsilon)} - u\| \leq K \epsilon^{\sigma }$, and we conclude that
	\begin{align*}
	u(x) - u_h(x) \leq C \epsilon^{\sigma}+ \frac{C}{\lambda}A_\varepsilon.
 %\delta^{2(2-\sigma)}k^2\frac{1}{\varepsilon^{4-\sigma}} + \frac{h^2}{k^2} \frac{1}{\varepsilon^{2-\sigma}}+ \delta^{4-\sigma}\frac{1}{\varepsilon^{4-\sigma}} + \frac{h^2}{\delta^{\sigma}} \frac{1}{\varepsilon^{2-\sigma}}\Big). 
	\end{align*}  
	Minimizing by taking $k^2= O\big(\frac{h\varepsilon}{\delta^{2-\sigma}}$\big), $\delta = O\big(h^\frac{1}{2}\varepsilon^\frac{1}{2}\big)$, and  $\varepsilon= O\big(h^{\frac{4-\sigma}{4+\sigma}}\big)$, leads to %the upper bound 
	\begin{align*}
	u(x) - u_h(x) \leq K h^{\frac{\sigma(4-\sigma)}{4+ \sigma}}.
	\end{align*}
	
	The lower bound on $u - u_h$ follows from a similar argument based on the solution $u_h$ of the scheme \eqref{eqn:approx:wdeg:nonlocal}. For technical reasons, we need to work with a different regularisation $u_h^{[\epsilon]}$ based on the fractional heat kernel, see the definition in \eqref{veps2}.
	Since $u_h$ solves \eqref{eqn:approx:wdeg:nonlocal}, we have 
	\begin{align*}
	\lambda \, u_h^{[\epsilon]} + \sup_{\A \in \mathcal{A}} \left\{ (f^{\A})^{[\epsilon]}(x) -\I^{\A}_h [u_h^{[\epsilon]}]  \right\} \leq 0 .
	\end{align*}
By \ref{C3} and Lemma \ref{lem:appndx_1}, $\|f^{\A} - (f^{\A})^{[\epsilon]}\| \leq C \epsilon^{\sigma}$, and then by Lemmas~\ref{lem:err_local_dif_sym_odd},~\ref{lem:local_trunc1} and~\ref{lem:trunc_err2}, 
	\begin{align*}
	 & \lambda \, u_h^{[\epsilon]} + \sup_{\A \in \mathcal{A}} \left\{ f^{\A}(x) -\I^{\A} [u_h^{[\epsilon]}]  \right\}  \leq K \epsilon^{\sigma} + \|\I^{\A}[u_h^{[\epsilon]}] - \I^{\A}_h[u_h^{[\epsilon]}]\|_0\\
	&  \leq K \epsilon^{\sigma} + C \Big( \big(\delta^{4-\sigma}  + \delta^{2(2-\sigma)} k^2 \big) \|D^4 u_h^{[\epsilon]}\|_0 +  \big(h^2 k^{-2} + h^{2} \delta^{-\sigma}\big) \|D^2 u_h^{[\epsilon]}\|_0 \Big)  := B_{\varepsilon}.
	\end{align*} 
	Hence  $u_h^{[\epsilon]} - \frac{C}{\lambda}B_{\epsilon}$ is a subsolution of equation \eqref{eqn:main:wdeg:nonlocal}, and  the comparison principle for \eqref{eqn:main:wdeg:nonlocal} then implies that $u_h^{[\epsilon]} - u \leq  \frac{C}{\lambda} \, B_{\epsilon}.$ 
	Therefore by Lemma \ref{thm:regularization_bd:wdeg:nonlocal}  
 and the bounds on $\|D^4 u_h^{[\epsilon]}\|_0$ and $\|D^2 u_h^{[\epsilon]}\|_0$ from Lemma \ref{lem:bd:der:numsol_reg:wdeg:nonlocal}, 
 we get	
\begin{align*}
 u_h - \,& u   \leq  C\Big(\delta+\varepsilon^\sigma + \delta^{2-\sigma}\varepsilon^{2(\sigma-1)}+ k^2\delta^{1-\sigma}  + \frac{h^2}{k^2}\delta^{\sigma-1}\Big) \\[0.2cm]
%\Big(\varepsilon^\sigma + \delta^{2-\sigma}\varepsilon^{2(\sigma-1)} + \delta + h + h^2\delta^{-1}\Big) \\
 &  +  C  \big(\delta^{4-\sigma}  + \delta^{2(2-\sigma)} k^2 \big) \frac{1+  (\delta^{4-\sigma}  + \delta^{2(2-\sigma)} k^{2} ) \frac{1}{\varepsilon^3} + (h^2 k^{-2} + h^2 \delta^{-\sigma} ) \frac{1}{\varepsilon} }{\varepsilon^{4-\sigma}} \\[0.2cm]
&  +  C \big(h^2 k^{-2} + h^{2} \delta^{-\sigma}\big) \frac{1+  (\delta^{4-\sigma}  + \delta^{2(2-\sigma)} k^{2} ) \frac{1}{\varepsilon^3} + (h^2 k^{-2} + h^2 \delta^{-\sigma} ) \frac{1}{\varepsilon} }{\varepsilon^{2-\sigma}} .
 \end{align*}
As in the proof of the upper bound, we now take $ k^2 = O\big(\delta^{\sigma} \big) $ so that
\begin{align*}
u_h-u \leq &\, K\Big(\varepsilon^\sigma +\delta^{2-\sigma}\varepsilon^{2(\sigma-1)} + \delta + h^2\delta^{-1} \Big) \\
 & \quad + C  \Big( \frac{h^2\delta^{-\sigma}}{\varepsilon^{2-\sigma}} +\frac{h^4\delta^{-2\sigma}}{\varepsilon^{3-\sigma}} + \frac{\delta^{4-\sigma}}{\varepsilon^{4-\sigma}} +2 \frac{h^2\delta^{4-2\sigma}}{\varepsilon^{5-\sigma}} + \frac{\delta^{8-2\sigma}}{\varepsilon^{7-\sigma}}\Big)\\
 =& A_1 + A_2.
\end{align*}
To continue note we can factor the second term,
$$A_2=C\frac1{\varepsilon^{1-\sigma}}\Big(\frac{h^2}{\varepsilon\delta^\sigma}+\frac{\delta^{4-\sigma}}{\varepsilon^{3}}\Big)\Big(1+\frac{h^2}{\varepsilon\delta^{\sigma}}+ \frac{\delta^{4-\sigma}}{\varepsilon^3}\Big).$$ 
Taking $\delta = O\big(h^\frac{1}{2}\varepsilon^\frac{1}{2}\big)$ as in the upper bound, we balance terms in $A_2$, and $A_2=\frac1{\varepsilon^{1-\sigma}}a(1+a)$ for $a^2=O(\frac{h^{4-\sigma}}{\varepsilon^{2+\sigma}})$. Finally (as for the upper bound) we take
 $\varepsilon= O\big(h^{\frac{4-\sigma}{4+\sigma}} \big)$. Then it is easy to check (for $h<1$) that $a=O(\varepsilon)$ and 
 $$A_2\leq O\Big(\frac a{\varepsilon^{1-\sigma}}\Big)=O(\varepsilon^\sigma)=O(h^{\frac{\sigma(4-\sigma)}{4+\sigma}}).$$
 In the remaining $A_1$ term, using $h\leq\delta\leq\varepsilon$ to estimate the 2nd and 4th terms, and a direct computation for the $\delta$-term, we find that the 2nd and 4th terms are $O(\varepsilon^\sigma)$ and $O(h)$, while $\delta=O(h^{\frac4{4+\sigma}})$. Since $\frac{\sigma(4-\sigma)}{4+\sigma}\leq \frac4{4+\sigma}\leq 1$, we conclude that
 \begin{align*}
u_h-u \leq C  h^{\frac{\sigma(4-\sigma)}{4+\sigma}}. %+ h^{\frac{(\sigma+1)(4-\sigma)}{4+\sigma}} + h^{\frac{\sigma(6-2\sigma)}{4+\sigma}}  \Big)  \leq C h^{\frac{\sigma(4-\sigma)}{4+\sigma}}.
\end{align*}	
This completes the proof of the theorem. \hspace{5.7cm} $\square$

\section{Proof of error bound for powers of discrete Laplacian } \label{sec:prf:flap}
We start with an analogous (uniform in $h$) bound as in Theorem \ref{thm:reg:approx:wdeg:nonlocal}.
\begin{thm}\label{thm:frac-bound-wdeg-flap}
Assume \ref{A1}-\ref{A5}, \ref{C1}-\ref{C3}, and $u_{h}$ solves \eqref{eqn:fraclap:aprrox}. Then there is $K>0$ independent of $h$ such that 
$$\|(-\Delta_h)^{\frac{\sigma}{2}}[u_h]\|_0 \leq K.$$
\end{thm}  
We omit the proof which is similar to the proof of Theorem \ref{thm:reg:approx:wdeg:nonlocal}, but simpler since we have no diffusion correction term in the approximation this time.  Next we state the analogous results to Lemmas \ref{lem:fracbd:conv:nonlocal} and \ref{lem:bd:der:numsol_reg:wdeg:nonlocal} for regularisations $u_h^{[\epsilon]}(x)$  by the fractional heat semigroup defined in \eqref{veps2}.
\begin{lem} \label{lem:reg:numsol_reg:wdeg:flap}
	Assume $\sigma>1$, \ref{A1}-\ref{A7}, \ref{C1}-\ref{C3}, $u_h$ solves \eqref{eqn:fraclap:aprrox}, and $u_h^{[\varepsilon]}$ is defined in \eqref{eqn:fraclap:aprrox}.
	Then there is $K>0$ independent of $h$ and $\varepsilon$ such that
	\begin{align} \label{reg:numsol_reg:wdeg:flap}
	\|u_h^{[\varepsilon]}\|_{1,\sigma-1} \leq K \Big( 1+ \frac{h^2}{\varepsilon^3}\Big), 
	\end{align}
and for $m\geq 2$,
$$\|D^m u_h^{[\varepsilon]}\|_0 \leq  \frac{K}{\varepsilon^{m-\sigma}} \Big(1+ \frac{h^2}{\varepsilon^3}\Big).$$
\end{lem} 

\begin{proof}
By Theorem \ref{thm:frac-bound-wdeg-flap} and
properties of $\tilde{K}^{\sigma}$, $\|(-\Delta_h)^{\frac{\sigma}{2}}[u_h^{[\varepsilon]}]\|_0\leq C\|(-\Delta_h)^{\frac{\sigma}{2}}[u_h]\|_0\leq CK$, and we conclude from the truncation error bound \eqref{fraclap_trunc} that
\begin{align}\label{imprpved:flap:bd:reg}
\|(-\Delta)^{\frac{\sigma}{2}}u_h^{[\varepsilon]}\|_0 \leq K_1\Big( \|(-\Delta_h)^{\frac{\sigma}{2}}[u_h] \|_0 + h^2 (\|D^4u_h^{[\varepsilon]}\|_0 + \|u_h^{[\varepsilon]}\|_0)\Big). 
\end{align}
Since $\|D^m u^{(\varepsilon)}_h\|_0 \leq \frac{C}{\varepsilon^{m-1}}\|u_h\|_{0,1}$ by Lemma \ref{lem:appndx_2}, estimate  \eqref{reg:numsol_reg:wdeg:flap} follows from the regularity estimate \cite[Theorem 1.1(a)]{Oton2016} by Ros-Oton and Serra for fractional Laplace operators. The second part follows from \eqref{reg:numsol_reg:wdeg:flap} and Lemma \ref{lem:appndx_2}.
\end{proof}
We give a version of Lemma \ref{thm:regularization_bd:wdeg:nonlocal} for powers of the discrete fractional Laplacian.
%which asserts the refined estimates for $\|u_h-u_h^{[\epsilon]}\|_0$ by choosing `\textit{fractional heat kernel}' as the mollifier function, $u_h^{[\epsilon]}$ is defined in Section \ref{sec_prf_gen}. 
\begin{lem}\label{thm:imprvdreg_bd:wdeg:flap}
Assume $\sigma>1$, $0<h\leq  \epsilon^{\frac{4-\sigma}2}$, \ref{A1}-\ref{A5}, \ref{C1}-\ref{C3}, $u_h$ solves \eqref{eqn:fraclap:aprrox}, and $u_h^{[\varepsilon]}$ is defined in \eqref{veps2}. Then
\begin{align}
\|u_h^{[\varepsilon]} -u_h\|_0 \leq K \Big(\varepsilon^{\sigma}\|(-\Delta_h)^{\frac{\sigma}{2}}[u_h]\|_0 + h^{\frac{2}{4-\sigma}} \|u_h\|_{0,1}\Big).
\end{align}
\end{lem}
\begin{proof}
The proof is similar to the proof of Lemma \ref{thm:regularization_bd:wdeg:nonlocal}. By definition \eqref{veps2}, $u_h^{[\varepsilon]}=S_r(u_h)$ where $\varepsilon=r^{\frac1\sigma}$ and $S_t$ is the fractional heat semigroup. Therefore using properties of heat kernels, estimate \eqref{imprpved:flap:bd:reg}, and the first part of Lemma \ref{lem:appndx_2}, we have
\begin{align*}
&|S_t(u_h) - S_s(u_h)| \leq \int_s^t K\Big( \|(-\Delta_h)^{\frac{\sigma}{2}}u_h\|_0 +\frac{h^2}{r^{\frac{3}{\sigma}}} \|u_h\|_{0,1}\Big) \, dr  \\
& \quad\leq K(t-s)\|(-\Delta_h)^{\frac{\sigma}{2}}u_h\|_0 + K h^2 \|u_h\|_{0,1} \Big(  \frac{1}{s^{\frac{3-\sigma}{\sigma}}} - \frac{1}{t^{\frac{3-\sigma}{\sigma}}} \Big),\\[0.2cm]
%\\
%& \leq K\Big( t \, \|(-\Delta_h)^{\frac{\sigma}{2}}u_h\|_0 + \frac{h^2}{s^{\frac{3-\sigma}{\sigma}}}\|u_h\|_{0,1}\Big). 
&|S_s(u_h) -u_h| = \Big|\int_{\rn}\Big(u_h(x-s^{\frac{1}{\sigma}} y) -u_h(x)\Big) K^{\sigma}(y) \, dy\Big| \leq  K  s^{\frac{1}{\sigma}} \|u_h\|_{0,1}.
\end{align*}
Moreover, 
\begin{align*}
|S_t(u_h) -u_h| \leq K\Big( t \, \|(-\Delta_h)^{\frac{\sigma}{2}}u_h\|_0 + \Big(\frac{h^2}{s^{\frac{3-\sigma}{\sigma}}} +  s^{\frac{1}{\sigma}} \Big)\|u_h\|_{0,1} \Big).
\end{align*}
The result now follows by taking $t= \varepsilon^{\sigma}$ and $s = h^{\frac{2\sigma}{4-\sigma}}$, noting that $s\leq t$ by the assumption that $h\leq \varepsilon^{\frac{4-\sigma}2}$.  
\end{proof}

\subsubsection*{Proof of Theorem \ref{thm:err_bd:flap:wdeg}}
 The case $\sigma<1$ uses no more than Lipschitz continuity of solutions and follows in straight forward way from the local truncation error bound in Lemma \ref{lem:trun_err_fraclap} and the regularisation/comparison arguments in the proof of Theorem \ref{thm:estimate_sym_main}. Therefore we focus on the case $\sigma>1$. The arguments are same as in the proof of Theorem \ref{thm:er_bound:wdeg}. 
To prove the upper bound on $u-u_h$, we  regularize equation \eqref{eqn:main3_wdeg} and use the truncation error bound \eqref{fraclap_trunc} to find that
\begin{align*}
&\lambda \, u^{(\epsilon)} + \sup_{\A \in \mathcal{A}} \left\{ f^{\A}(x) + a^{\A} (-\Delta_h)^{\frac{\sigma}{2}} u^{(\epsilon)}  \right\} \\
& \leq   \sup_{\A \in \mathcal{A}} \|f^{\A} - (f^{\A})^{(\epsilon)}\|_0 +  \sup_{\A \in \mathcal{A}} a^{\A} \|(-\Delta_h)^{\frac{\sigma}{2}} u^{(\epsilon)} - (-\Delta)^{\frac{\sigma}{2}} u^{(\epsilon)}\|_0\\
& \leq  K \epsilon^\sigma + C h^2 \Big(\|D^4 u^{(\epsilon)} \|_0 + \|u^{(\epsilon)}\|_0 \Big) .
\end{align*}
Hence $u^{(\epsilon)} - \frac{C}{\lambda}\big( \epsilon^\sigma +  h^2 (\|D^4 u^{(\epsilon)} \|_0 + \|u^{(\epsilon)}\|_0 )\big)$ is a subsolution of equation \eqref{eqn:fraclap:aprrox}. 
By  the comparison principle for \eqref{eqn:fraclap:aprrox}, regularity of $u$ given by Theorem \ref{thm:regularity:vis_soln:wdeg:nonlocal}, and the bounds given by Lemma  \ref{lem:dif:conv}, we have 
\begin{align*}
u(x) - u_h(x) \leq K \Big( \epsilon^{\sigma} + \frac{h^2}{\epsilon^{4-\sigma }} \Big). 
\end{align*}  
We optimize the right hand side by choosing $\epsilon = O\big(h^{\frac{1}{2}}\big)$ and get
\begin{align*}
u(x) - u_h(x) \leq K h^{\frac{\sigma}{2}}.\\
\end{align*}

To prove the lower bound we mollify/regularize the scheme \eqref{eqn:fraclap:aprrox} using the fractional heat semigroup. Then by Lemma \ref{lem:appndx_1} for $f^{\A}$ and the truncation error \eqref{fraclap_trunc}, 
\begin{align*}
 \lambda \, u_h^{[\epsilon]} + \sup_{\A \in \mathcal{A}} \left\{ f^{\A}(x) + a^{\A}(-\Delta)^{\frac{\sigma}{2}} u_h^{[\epsilon]}  \right\} \leq C\epsilon^{\sigma} + C h^2 \Big(\|D^4 u_h^{[\epsilon]} \|_0 + \|u_h^{[\epsilon]}\|_0 \Big) . 
 \end{align*}
Therefore $u_h^{[\epsilon]} - \frac{C}{\lambda}\big(\epsilon^{\sigma} +  h^2\|D^4 u^{[\epsilon]} \|_0 + h^2\|u^{[\epsilon]}\|_0 \big)$ is a subsolution of equation \eqref{eqn:main3_wdeg}, and comparison for \eqref{eqn:main3_wdeg} then yields
$$u_h^{[\epsilon]} - u \leq  \frac{C}{\lambda}\Big(\epsilon^{\sigma} +  h^2\|D^4 u_h^{[\epsilon]} \|_0 + h^2\|u_h^{[\epsilon]}\|_0 \Big).  $$ 
Then by Lemma
\ref{thm:imprvdreg_bd:wdeg:flap} (needs $h\leq \varepsilon^{\frac{4-\sigma}{2}}$) and the $\|D^4 u_h^{[\epsilon]}\|_0$-bound of Lemma
\ref{lem:reg:numsol_reg:wdeg:flap},  
\begin{align*}
u_h - u & \leq  
%C\Big(\epsilon^{\sigma}  + h^{\frac{2}{4-\sigma}} +  \frac{h^2}{\epsilon^{4-\sigma}} \Big(1+ \frac{h^2}{\epsilon^{3}}\Big)  \Big)  = 
C\Big(\epsilon^{\sigma}  + h^{\frac{2}{4-\sigma}} +  \frac{h^2}{\epsilon^{4-\sigma}} +  \frac{h^4}{\epsilon^{7-\sigma}} \Big) .
\end{align*}
Optimizing in $\epsilon$ by choosing $\epsilon= O \big(h^{\frac{1}{2}}\big)$, we get the final estimate 
$$u_h-u \leq K \big( h^{\frac{\sigma}{2}} + h^{\frac{2}{4-\sigma}}\big). $$
The result now follows since $\frac{2}{4-\sigma}> \frac{\sigma}{2}$ for $\sigma>1$ and $h=\varepsilon^2\leq \varepsilon^{\frac{2}{4-\sigma}}$ (for $h<1$).  \hfill $\square$
 
\section{Extensions} \label{sec:exten}

In this section we discuss two related extensions of our previous results: (i) to nonlocal HJB equations with drift/advection terms, and %to the case 
(ii) to jump diffusions with nonsymmetric singular parts in the sense that we drop condition \ref{A7}. Consider
\begin{align}\label{eqn:main4}
 &  \sup_{\A\in\mathcal{A}} \big\{ f^{\A}(x) +c^{\A}(x) u(x)  - b^{\A} \cdot \grad u(t,x)- \I^{\A}[u](x) \big\}  = 0,  \mbox{\quad in} \ \rn, 
 \end{align}
where $b^{\A}\in \rn$ and a modified version of \ref{A2} holds:
\begin{Assumptions1}
  \setcounter{enumi}{1}
\item\label{A2'} There is a $K>0$ such that 
$$  \|f^{\A}\|_{1}+ \|c^{\A}\|_{1} + |b^\A|+\|\eta^{\A}\|_0 \leq
  K\quad\text{for}\quad \A\in \mathcal{A}. $$  
\end{Assumptions1}
Under assumptions \ref{A1}, \ref{A2'}, \ref{A3}, \ref{A4} equation
\eqref{eqn:main4} is well-posed, comparison holds, and the $C_b$ and
Lipschitz bounds of Proposition \ref{thm:viscosity_exist} hold. The
proof is the same as for Proposition \ref{thm:viscosity_exist}.
Note that $x$-independent $b^\A$ is consistent with $x$-independent $\eta^\A$ in \eqref{n_local:term} and simplifies the
presentation below. 

Dropping \ref{A7} means that
$$\tilde{b}^{\A, \delta}:=\int_{\delta<|z|<1} \eta^{\A}(z) \, \nu_{\A}(dz) \neq 0,$$ 
and there is
a new drift term in our equation. We can write the nonlocal term as
\begin{align*}
\I^{\A}[\phi](x) = \I^{\A}_{\delta} [\phi](x) + \I^{\A,\delta}[\phi](x) - \tilde{b}^{\A, \delta}\cdot \nabla \phi(x), 
\end{align*}
where 
%$\tilde{b}^{\A, \delta}:= \int_{\delta<|z|<1} \, \eta^{\A}(z)  \nu_{\A}(dz) $ and  
$\I^{\A}_{\delta} $, $\I^{\A,\delta}$ are defined in Section \ref{monotone scheme}. The term $\tilde{b}^{\A, \delta}$ is bounded under a $C^{1,1}$ condition for $\eta^\A$ at $z=0$, a uniform in $\A$ version of assumption \ref{C1} (ii): \smallskip
\begin{Assumptions}
\setcounter{enumi}{7}
\item \label{A8} There is $K>0$ such that  
$$|\eta^{\A}(z) - 2\eta^{\A}(0) - \eta^{\A}(-z)|\leq K |z|^2 \qquad \mbox{for} \qquad |z|<1, \quad \A\in\mathcal A.$$
\end{Assumptions}
This assumption is satisfied in most applications. The next result is a version of Lemma \ref{lem:err_local_dif_sym_odd} without \ref{A7}.
\begin{lem}\label{lem:err_local_dif_sym_ex}
Assume \ref{A1}, \ref{A2'}, \ref{A3} - \ref{A6} and $\delta\in(0.1)$. 
%and $\phi\in C^3_b(\rn)$.
\begin{itemize}
\item[(i)]There is $K>0$ independent of $\delta,\A,\phi$ such that  \begin{align}\label{condn:small-difu}
|\I^{\A}_{\delta}[\phi]- tr[a^{\A}_{\delta}D^2 \phi] | \leq K\delta^{3-\sigma} \|D^3\phi\|_0.
\end{align}
\item[(ii)] If also \ref{A8} holds, there is $C>0$ independent of $\delta,\A$  such that $|\tilde{b}^{\A,\delta}| \leq C$.
\end{itemize}
\end{lem}
\begin{proof}
(i) \ The proof is similar to the proof of Lemma \ref{lem:err_local_dif_sym_odd}. After a Taylor expansion of $\phi$, we find that
\begin{align*}
\I^{\A}_{\delta} [\phi](x) = tr[a^{\A}_{\delta} D^2 \phi] + Err_{1,\delta}, 
\end{align*}
where $Err_{\delta}= \frac{|\beta|}{\beta!}\sum_{|\beta|=3}\big[\int_{|z|<\delta} \int_0^1(1-s)^{|\beta-1|} D^{\beta}\phi(x+s\eta^{\A}(z)) \eta^\A(z)^{\beta}\, ds\, \nu_{\A}(dz)\big] $ and $a^{\A}_{\delta}$ is defined in Lemma \ref{lem:err_local_dif_sym_odd}. By \ref{A6} we have $|Err_{1, \delta}|\leq C \delta^{3-\sigma}\|D^3\phi\|_0.$
\smallskip

\noindent (ii) Since $\eta^{\A}(0)=0$ by \ref{A3}, assumptions \ref{A5} and \ref{A8} lead to
\begin{align*}
&\big|\tilde{b}^{\A, \delta}\big| =\frac{1}{2} \Big|\int_{\delta<|z|<1} \big(\eta^{\A}(z)+ \eta^{\A}(-z)\big) \, \nu_{\A}(dz) \Big|  \leq \, K \int_{\delta<|z|<1}|z|^2 \, \nu_{\A}(dz).  
\end{align*}  
By \ref{A4}, this completes the proof.
\end{proof}

Following the approach of Section \ref{monotone
  scheme}, to discretize \eqref{eqn:main4} we first approximate small jumps by a diffusion. This leads to equation \eqref{eqn:main:apprx} with a redefined operator $\mathcal{L}_\delta^{\A}$ to
account for the drift:
\begin{align} \label{local_term_ex}
\mathcal{L}_\delta^{\A}[\phi](x) :=  tr[a^{\A}_{\delta} D^2 \phi](x) + b^{\A}_{\delta}\cdot \grad \phi(x), \qquad b^{\A}_{\delta}= b^{\A}-\tilde{b}^{\A, \delta},
\end{align}
where $b^{\A}_{\delta}$ is bounded under \ref{A2'} and \ref{A8}. Then we approximate $\mathcal{L}_\delta^{\A}$ by
 \begin{align}\label{local_discrete_ex}
\bar{\mathcal{L}}^{\A}_{\delta,k,h} \phi= \mathcal{L}^{\A}_{\delta,k,h} [\phi] + b_k^{\A,+} \, \delta_{h,e_k} \phi + b_k^{\A,-} \, \delta_{h,-e_k} \phi ,
\end{align}
where $\mathcal{L}^{\A}_{\delta,k,h} $ is defined in \eqref{local_epprox_term}, $e_{k}$ are basis vectors in $\rn$, 
%where $1\leq k \leq N$, 
$b^{\A}_{\delta}= (b^{\A}_1, \cdots , b^{\A}_N)$, and  
$$ \delta_{h,l}u(x)= \frac{u(x+hl)-u(x)}{h}\qquad\text{for}\qquad l\in\rn,\neq0.$$
Here the drift term is discretized by an upwind finite difference method\footnote{This is just an example, many other monotone discretizations would also work here, also SL schemes.}, and the total discretization is still monotone. We estimate the truncation error next.
\begin{lem}\label{lem:local_trunc_ex}
Assume \ref{A1}, \ref{A2'}, \ref{A3}-\ref{A6}, \ref{A8}, $\phi\in C^4(\rn)$, and $\mathcal{L}_\delta^{\A}$ and $\bar{\mathcal{L}}^{\A}_{\delta,k,h}$ are defined by \eqref{local_term_ex} and \eqref{local_discrete_ex}. 
%\begin{itemize}
%\item[(i)] Then we have
% 
Then there is $K$ independent of $h,k,\delta$ such that
%\end{itemize} 
 \begin{align}\label{local_trunc1_drift_ex}
 \big| \bar{\mathcal{L}}^{\A}_{\delta,k,h} [\phi] - \mathcal{L}_\delta^{\A} [\phi]  \big| \leq K \Big( h  \|D^2 \phi\|_0 +\delta^{2(2-\sigma)} k^2 \|D^4\phi\|_0 + \frac{h^2}{k^2} \|D^2 \phi\|_0\Big).  
\end{align} 
\end{lem} 
\begin{proof}The first term on the right hand side of  %\eqref{local_trunc1_ex} and 
\eqref{local_trunc1_drift_ex} is classical and due to the approximation of the drift. The remaining terms come from Lemma \ref{lem:local_trunc1}.
\end{proof}

The numerical scheme for equation \eqref{eqn:main4} is defined by 
\begin{align}\label{approx:eqn1_ex}
 \sup_{\A \in \mathcal{A}} \Big\{ f^{\A}(x) + c^{\A}(x)u(x)- \bar{\mathcal{L}}^{\A}_{\delta,k,h} [u](x) - \mathcal{I}_h^{\A, \delta}[u](x) \Big\} =0 \quad \text{in} \quad \rn, 
\end{align}  
 where $\bar{\mathcal{L}}^{\A}_{\delta,k,h}$ and $\mathcal{I}_h^{\A, \delta}$ are given by \eqref{local_discrete_ex} and \eqref{eq:discrete_nonlocal}. This is a consistent, monotone, and $L^\infty$-stable scheme. In the strongly degenerate case, an error estimate given by the next result.\ee
\begin{thm}\label{thm:estimate_sym_main_ex}
Assume $\sigma\in(0,2)$, $h,k\in(0,1)$, $\delta\geq h$, \ref{A1}, \ref{A2'}, \ref{A3}-\ref{A6}, \ref{A8}, $u$ and $u_h$ solves \eqref{eqn:main4} and  \eqref{approx:eqn1_ex}. 
\begin{itemize}
\item[(a)] 
%If $k^2= O \big(\frac{h \epsilon}{\delta^{2-\sigma}}\big)$ %$\epsilon = O \big(h\delta^{\frac{-\sigma}{2}}\big)$ 
If $k^2=O(\frac {h^2}{\delta^{2-\frac\sigma2}})$, then there is $C>0$ such that
\begin{align}\label{estimate_sym_main_gen_ex}
|u- u_{h}| \leq \left\{
       \begin{array}{ll}
      C \, h^{\frac{1}{2}} \quad &\mbox{for} \quad 0< \sigma \leq \frac{6}{5} \quad\text{and}\quad \delta= O \big(h^{\frac{1}{\sigma}}\big)\\[0.2cm]
      C \, h^{\frac{2(3-\sigma)}{6+ \sigma}} \quad &\mbox{for} \quad \frac{6}{5}< \sigma <2 \quad\text{and}\quad \delta= O \big(h^{\frac{6}{6+\sigma}}\big).
       \end{array}
  \right.
\end{align}
\item[(b)] When \ref{A7} holds and $k^2=O(\frac {h^2}{\delta^{2-\frac\sigma2}})$, then there is $C>0$ such that  
\begin{align}\label{estimate_sym_main_ex}
|u- u_{h}| \leq \left\{
       \begin{array}{ll}
      C h^{\frac{1}{2}} \quad &\mbox{for} \quad 0< \sigma \leq \frac{4}{3} \quad\text{and}\quad \delta= O \big(h^{\frac{1}{\sigma}}\big)\\[0.2cm]
      C h^{\frac{4-\sigma}{4+ \sigma}} \quad &\mbox{for} \quad \frac{4}{3}< \sigma <2 \quad\text{and}\quad \delta= O \big(h^{\frac{4}{4+\sigma}}\big).
       \end{array}
  \right.
\end{align}
\end{itemize}
\end{thm}

\begin{rem} 
\noindent (a) \,  When $\sigma\leq 1$, the error can can not be better than $\mathcal{O} (h^{\frac{1}{2}})$ 
because of the (local) drift term in \eqref{eqn:main4}. In this case the diffusion correction does not improve the rate as it did in Section \ref{monotone scheme}. \smallskip

\noindent (b) \, Under assumption \ref{A8}, we get improved convergence rates for any $\sigma>1$, see Theorem \ref{thm:estimate_sym_main_ex} (b). The rate approaches $\frac{1}{3}$
as $\sigma\to2$, compared to $\frac14$ in part (a). 
%Also, by the definite symmetric structure on the nonlocal term (c.f. \ref{A5}, \ref{A7}) further improvement on error bound has been noticed in Theorem \ref{thm:estimate_sym_main_ex}(c) and as $\sigma$ close to $2$ the rate asymptotically approaches to $\mathcal{O} (h^{\frac{1}{3}})$. 
%\smallskip

%\noindent (c) \, It is possible to extend the results for weakly non-degenerate case to the equation with drift term as well.

\end{rem}

\noindent \textit{Sketch of proof:} The proof is similar to the proof of Theorem \ref{thm:estimate_sym_main}, we only explain the main differences. 
In view of Lemmas \ref{lem:err_local_dif_sym_odd} and \ref{lem:err_local_dif_sym_ex}, 
replacing Lemma \ref{lem:local_trunc1} by Lemma \ref{lem:local_trunc_ex} when estimating \eqref{prf_err1_new}, the constant
$M_{\epsilon, \delta}$ in \eqref{prf_err_M1} gets a $O(\frac h\epsilon)$ contribution from the drift and becomes
\begin{align}\label{prf_err_M1_ex}
  M_{\epsilon,\delta} \:=\left\{
     \begin{array}{@{\space}l@{\thinspace}l}
     %  \delta^{3-\sigma} \, \frac{1}{\epsilon^2} + h \Gamma(\sigma, \delta)\, \frac{1}{\epsilon} + k^2   \delta^{2(2-\sigma)}  \frac{1}{\epsilon^3} + \frac{h^2}{k^2} \frac{1}{\epsilon} + \frac{h^2}{\delta^{\sigma}} \, \frac{1}{\epsilon} \, \,  & \text{for part (a)},\\\\
       \delta^{3-\sigma} \,\frac{1}{\epsilon^2} + h \, \frac{1}{\epsilon} + k^2  \, \delta^{2(2-\sigma)} \,\frac{1}{\epsilon^3}+ \frac{h^2}{k^2} \frac{1}{\epsilon} + \frac{h^2}{\delta^{\sigma}} \, \frac{1}{\epsilon} \quad & \text{for part (a)}, \\\\ 
       \delta^{4-\sigma} \, \frac{1}{\epsilon^3} + h \, \frac{1}{\epsilon} + k^2  \, \delta^{2(2-\sigma)} \, \frac{1}{\epsilon^3} + \frac{h^2}{k^2} \frac{1}{\epsilon} + \frac{h^2}{\delta^{\sigma}} \, \frac{1}{\epsilon}  & \text{for part (b)}.
     \end{array}
   \right.
\end{align}
In case (a) the nonlocal operator is not symmetric, so 
we have used 
Lemma \ref{lem:err_local_dif_sym_ex} (i) to get the first term. 
%
%
%Then we optimize the corresponding one sided bounds on $u-u_h$ in \eqref{prf:upperbd:gen1} and \eqref{prf_err1_5}, first taking $k^2= O \big(\frac{h \epsilon}{\delta^{2-\sigma}}\big)$
%For $\sigma>1$ we take
%and $\epsilon = O \big(\frac h{\delta^{\frac{\sigma}{2}}}\big)$ to get the two sided bound (compare with \eqref{prf_err_M2})
%\begin{align}\label{prf_err_M2_ex}
% |u-u_h| \leq \left\{
%     \begin{array}{@{\space}l@{\thinspace}l}
%      C \big(\delta^{3} h^{-2} + \delta^{\frac{\sigma}{2}} + \delta^2 h^{-1} + \frac{h}{\delta^{\frac{\sigma}{2}}} \big) \qquad & \text{for part (a)}, \\ \\
 %        C \big(\delta^{4+\frac{\sigma}{2}} h^{-3} + \delta^{\frac{\sigma}{2}} + \delta^2 h^{-1} + \frac{h}{\delta^{\frac{\sigma}{2}}} \big) \qquad  & \text{for part (b)}.\\
  %   \end{array}
   %\right.
%\end{align}
%For part (a), the rate \eqref{estimate_sym_main_gen_ex} follows by choosing $\delta$ optimally as $\delta= O \big(h^{\frac{1}{\sigma}}\big)$ for $0< \sigma \leq \frac65$ %SOMETHING ELSE? $\frac65<\sigma<\frac43$, and $\delta = O \big( h^{\frac{6}{6+ \sigma}}\big)$  for $\frac{6}{5}\leq \sigma <2$.
% For part (b), the convergence rate  \eqref{estimate_sym_main_ex} is observed by choosing $\delta$ optimally  as $\delta= O \big(h^{\frac{1}{\sigma}}\big)$ for $0< \sigma \leq \frac{4}{3}$ and $\delta = O \big(h^{\frac{4}{4+ \sigma}}\big)$ for $\frac{4}{3}\leq \sigma <2$. $\hfill\square$
%
By \eqref{prf:upperbd:gen1} and \eqref{prf_err1_5} we get 
$|u-u_h|\leq C (\epsilon + M_{\epsilon,\delta})$ and optimize with
respect to $k, \delta,$ and $\epsilon$. First we take $k^2= O \big(\frac{h \epsilon}{\delta^{2-\sigma}}\big)$,
%For $\sigma>1$ we take 
then using $h\leq \delta$, we take $\delta = O \big(h^{\frac{2}{3}}\epsilon^{\frac{1}{3}}\big)$ for part (a) and $\delta = O \big(h^{\frac{1}{2}}\epsilon^{\frac{1}{2}}\big)$ for part (b) to get
\begin{align}\label{prf_err_M2_ex}
 |u-u_h| \leq \left\{
     \begin{array}{@{\space}l@{\thinspace}l}
      C \big( h^{\frac23(3-\sigma)}\epsilon^{-\frac{1}{3}(3+\sigma)} +h\frac{1}{\epsilon} +\epsilon \big) \qquad & \text{for part (a)}, \\ \\
         C \big( h^{\frac12(4-\sigma)}\epsilon^{-\frac12(2+\sigma)} +h\frac{1}{\epsilon} +\epsilon \big) \qquad  & \text{for part (b)}.\\
     \end{array}
   \right.
\end{align}
For part (a), the rate \eqref{estimate_sym_main_gen_ex} follows by choosing $\epsilon= O\big(\max\big\{h^{\frac{1}{2}}, h^{\frac{2(3-\sigma)}{6+ \sigma}}\big\}\big)$, i.e. $\epsilon= O \big(h^{\frac{1}{2}}\big)$ for $0< \sigma \leq \frac56$,  and $\epsilon = O \big( h^{\frac{2(3-\sigma)}{6+ \sigma}}\big)$  for $\frac{5}{6}\leq \sigma <2$.
For part (b), the convergence rate  \eqref{estimate_sym_main_ex} is observed by choosing $\epsilon$ optimally  as $\epsilon =O\big(\max\big\{h^{\frac{1}{2}}, h^{\frac{4-\sigma}{4+ \sigma}}\big\}\big)$, i.e. $\epsilon= O \big(h^{\frac{1}{2}}\big)$ for $0< \sigma \leq \frac{4}{3}$ and $\epsilon = O \big(h^{\frac{4-\sigma}{4+ \sigma}}\big)$ for $\frac{4}{3}\leq \sigma <2$. $\hfill\square$

\end{document}